\documentclass{article}

\usepackage[normalem]{ulem}
\usepackage[utf8]{inputenc}
\usepackage{amsfonts,amsmath,amssymb,amsthm,graphicx,lineno,caption,float,subcaption}
\usepackage{geometry} % Make the margins a little smaller
\usepackage{xcolor}
\usepackage{thm-restate}
\usepackage{hyperref} % For preprint URLs
\usepackage[capitalise, noabbrev, nameinlink]{cleveref}
\hypersetup{pdfborder={0 0 0}}
\usepackage{algorithm2e}
\newtheorem{theorem}{Theorem}[section]
\newtheorem{lemma}[theorem]{Lemma}
\newtheorem{prop}[theorem]{Proposition}
\newtheorem{obs}[theorem]{Observation}

\newtheorem{question}[theorem]{Question}

\newtheorem{cor}[theorem]{Corollary}
\newtheorem{definition}[theorem]{Definition}
\newtheorem{claim}{Claim}[theorem]
\crefname{claim}{Claim}{Claims}
\crefname{equation}{Equation}{Equations}
\crefname{obs}{Observation}{Observations}
\newenvironment{subproof}[1][\proofname]{%
  \begin{proof}[#1]%
}{%
  \end{proof}%
}
\newcommand{\usp}[1]{\operatorname{USP}{\mathop{}\!(#1)}}
\DeclareMathOperator{\diam}{diam}
\newcommand{\uspc}[1]{\operatorname{\overline{USP}}{\mathop{}\!(#1)}}
\newcommand{\uspcf}[1]{\lfloor\operatorname{\overline{USP}}\rfloor{\mathop{}\!(#1)}}   % following the notation in "Parameters related to tree-width, zero forcing, and maximum nullity of a graph" the floor/ceiling notation should just be around the \uspc because it indicates the minor-monotone version of a parameter, not the floor of a number.
\usepackage{pgf,tikz}
\tikzstyle{vertex}=[circle, draw, inner sep=0pt, minimum size=6pt]
\newcommand{\vertex}{\node[vertex]}
\newcommand{\equal}{=}
\newcommand{\M}{\mathrm{M}}
\newcommand{\Z}{\mathrm{Z}}

\definecolor{Green}{RGB}{34, 139, 34}

\title{A combinatorial bound on the number of distinct eigenvalues of a graph}
\author{Sarah Allred\thanks{Department of Mathematics \& Statistics, University of South Alabama,  Mobile, AL 36688, USA (sarahallred@southalabama.edu).} \and
Craig Erickson\thanks{Amazon, Minneapolis, MN, USA (craigeri@amazon.com). Work done while at Hamline University.} \and
Kevin Grace\thanks{Department of Mathematics \& Statistics, University of South Alabama,  Mobile, AL 36688, USA (kevingrace@southalabama.edu).} \and
H.~Tracy Hall\thanks{Hall Labs LLC, Provo, UT 84606, USA (H.Tracy@gmail.com).} \and
Alathea Jensen\thanks{Department of Mathematics and Computer Science, Susquehanna University, Selinsgrove, PA 17870, USA (jensena@susqu.edu).}
}

\begin{document}

\maketitle

\begin{abstract}
The smallest possible number of distinct eigenvalues of a graph $G$, denoted by $q(G)$, has a combinatorial bound in terms of unique shortest paths in the graph. In particular, $q(G)$ is bounded below by $k$, where $k$ is the number of vertices of a unique shortest path joining any pair of vertices in $G$. Thus, if $n$ is the number of vertices of $G$, then $n-q(G)$ is bounded above by the size of the complement (with respect to the vertex set of $G$) of the vertex set of the longest unique shortest path joining any pair of vertices of $G$. The purpose of this paper is to commence the study of the minor-monotone floor of $n-k$, which is the minimum of $n-k$ among all graphs of which $G$ is a minor. Accordingly, we prove some results about this minor-monotone floor. 
\end{abstract}

\section{Motivation and background}

The Inverse Eigenvalue Problem for a Graph (IEPG) starts with a pattern of
zero and nonzero constraints
for a real symmetric matrix, described by a graph $G$ on $n$ vertices, and asks what spectra are possible
for the set of $n \times n$ matrices $\mathcal{S}(G)$ that exhibit this pattern.
%The quantity $q(A)$ measures the number of distinct eigenvalues of $A$, which is also the degree of the minimal polynomial of $A$.
%The graph parameter $q(G)$, (which has been studied, for example in \cite{F10,MinimumDistinct,Analysis,Nordhaus}),
%measures the minimum value of $q(A)$ over all matrices $A \in \mathcal{S}(G)$.
Two narrower questions about the possible spectra of matrices in $\mathcal{S}(G)$ yield important matrix-theoretic graph parameters that motivate what is studied here. Each of these parameters has a natural combinatorial bound.
%Special cases of spectral questions for a matrix $A \in \mathcal{S}(G)$
Firstly, the highest possible nullity, denoted by $\M(G)$, is also equal to the maximum multiplicity that can be obtained by any eigenvalue, and this is bounded by a process called zero forcing, giving $\M(G) \le \Z(G)$, where $\Z(G)$ is the zero forcing number of $G$. (See, for example, \cite{HLS2022} for the definition of zero forcing number.)
Secondly, the lowest possible number of distinct eigenvalues is denoted by $q(G)$, and this is bounded by the number of vertices in a unique shortest path, which we denote by $\usp{G}$, thus giving $q(G) \ge \usp{G}$ \cite[Theorem 3.2]{MinimumDistinct}. (The parameter $q(G)$ has been studied, for example, in \cite{F10,MinimumDistinct,Analysis,Nordhaus}.)

{\color{black}Note that these combinatorial bounds are in opposite directions, and that they bound matrix parameters that also tend in opposite directions: Zero forcing sets provide an \emph{upper} bound on a matrix parameter that
\begin{itemize}
    \item tends to increase with more eigenvalue coincidences, and that
    \item tends to increase in denser graphs,
\end{itemize}
whereas unique shortest paths provide a \emph{lower} bound on a matrix parameter that
\begin{itemize}
    \item tends to decrease with more eigenvalue coincidences, and that
    \item tends to decrease in denser graphs.
\end{itemize}

One way to harmonize the directions of these inequalities would be to
define $\mathrm{mr}(G) = n - \M(G)$ and $\mathrm{tri}(G) = n - \Z(G)$
(these are in fact existing graph parameters, called respectively the \emph{minimum rank} of $G$ and the \emph{triangle number} of $G$), and to express the first inequality as $\mathrm{mr}(G) \ge \mathrm{tri}(G)$, which now runs in the same direction as the second inequality.
The complemented inequality gives a completely equivalent statement---but not, as it turns out, an equally useful statement,
at least not within the context of minor monotonicity.
Once we start taking minors, including operations such as contraction that change the number of vertices, it is $\M(G)$ rather than $\mathrm{mr}(G)$ that has a natural minor-monotone variant, and it is $\Z(G)$ rather than $\mathrm{tri}(G)$ that has a natural minor-monotone variant.
This suggests that it may be preferable to harmonize directions not by complementing the first inequality with respect to $n$, but rather by complementing the parameters in the second inequality.
For these parameters, unlike for $\M(G)$ and $\Z(G)$, the complements do not have existing names, and rather than introducing a profusion of names, we adopt a notational convenience
that abuses a common notation for graph complementation.
For a graph $G$, we follow the convention that $\overline{G}$ denotes the complement of $G$, whose edges are precisely the non-edges of $G$.
The same notation applied to the name of a graph parameter, rather than to a graph,
will however denote numerical complementation with respect to $n$.
This gives, for example, $\M(G) = \overline{\mathrm{mr}}(G)$ and
$\Z(G) = \overline{\mathrm{tri}}(G)$.
(The only purpose in mentioning the parameters $\mathrm{mr}(G)$ and $\mathrm{tri}(G)$ was
to provide an illustration of numerical complementation; having served that purpose, they will be of no further use to us.)
In this notation, the complemented second inequality becomes
\[
\overline{q}(G) \le \uspc{G},
\]
now expressed in terms of parameters that one might hope to have natural minor-monotone variants.

%As stated, these combinatorial bounds are in opposite directions, but in fact the matrix parameter that will be of interest is $n - q(G)$, which can be thought of as the maximum number of ordered eigenvalue coincidences, and is therefore called the \emph{maximum spectral equality}.
The graph parameter $q(G)$ is the maximum, over $A \in \mathcal{S}(G)$, of an identically named matrix parameter $q(A)$ that counts the number of distinct eigenvalues.
The numerically complemented matrix parameter $\overline{q}(A)$ has a straightforward interpretation:
Given the ordered spectrum
\[
\lambda_1 \le \lambda_2 \le \dots \le \lambda_n
\]
of $A$,
$\overline{q}(A)$ is precisely the number of eigenvalue coincidences, the number of times
that equality rather than strict inequality holds.
Whereas $\M(A)$ and $\M(G)$ only count ($1$ more than the number of) eigenvalue coincidences concentrated at a single eigenvalue, $\overline{q}(A)$ and $\overline{q}(G)$ count all coincidences, which might be taken as a broader measure of how tightly constrained the spectrum can be.
To give a simple example, both the claw graph $K_{1, 3}$ and the $4$-cycle $C_4$ achieve
maximum multiplicity
$\M(G)=2$, but a tree must always have simple eigenvalues at the extremities, and so
for the tree we have only $\overline{q}(K_{1, 3})=1$, but
for the cycle we can achieve $\overline{q}(C_4) = 2$.
The number $\overline{q}(G)$ will be called the \emph{maximum spectral equality} of the graph.
The graph parameter $\uspc{G}$ will also be given a combinatorial interpretation as
the \emph{spectator number} of a graph.
%Expressed this way, the combinatorial bound is $n - q(G) \le n - \usp{G}$, involving a combinatorial quantity $n - \usp{G}$ that will be called the \emph{spectator number} of $G$ and will be denoted by $\uspc{G}$.
}

The motivation for this project concerns the behavior of the four graph parameters $\M(G)$, $\Z(G)$,
%$n - q(G)$, and $n - \usp{G}$
{\color{black} $\overline{q}(G)$, and $\uspc{G}$}
with respect to subgraphs, and more generally with respect to graph minors.
{\color{black} Generally speaking, graphs that are smaller in the minor ordering impose more constraints on matrix entries, and allow less flexibility in trying to produce a constrained spectrum.}
The set $\mathcal{S}(G)$ is a manifold, an open subset of a vector space, and the process of removing an edge from the graph, forcing a zero entry in the matrix, collapses a coordinate direction of the enveloping vector space, reducing the dimension of $\mathcal{S}(G)$ by one.
% The process of removing an edge from a graph changes a nonzero entry in the matrix to zero, which reduces the dimension (as a manifold) of the space of matrices that belong to $\mathcal{S}(G)$.
One might expect, generically, that removing degrees of freedom in the matrix would not allow greater nullity and would also not allow more eigenvalue coincidences.
In other words, one might expect both $\M(G)$ and $n - q(G)$ to
be weakly decreasing as edges are deleted, and might expect a similar subgraph monotonicity
for the related combinatorial bounds $\Z(G)$ and $n - \usp{G}$.
% 
% Removing edges from a graph imposes additional linear constraints in the corresponding matrix,
% and so one would expect generically that removing edges would constrict rather
% than expand the set of possible nullities (decreasing $\M(G)$), and would constrict rather than
% expand the set of possibilities for the number of distinct eigenvalues (increasing $q(G)$).
In most cases this expected behavior is observed to hold, but there are counterexamples---an extreme example being that the graph with the fewest edges of all on $n$ vertices, the edgeless graph $\overline{K_n} = nK_1$, is the only graph that achieves a value as high as $n$ for either $\M(G)$ or $\Z(G)$ and is the only graph that achieves a value as high as $n - 1$ for either $n - q(G)$ or $n - \usp{G}$.
(This extreme example is also maximally disconnected, but there also exist, for any of these four parameters, connected counterexample graphs to which an edge can be added while decreasing the value of the parameter. See, e.g., Example 2.9 in \cite{Parameters} for an example for $\M(G)$ and $\Z(G)$, and see, e.g., Figures 6.1 and 6.2 of \cite{MinimumDistinct} for an example for $n-q(G)$ and $n-\usp{G}$.) 

For three of these four parameters, namely $\M(G)$, $\Z(G)$, and $n - q(G)$, there is an established variant of the graph parameter that not only exhibits the generically expected behavior of weakly decreasing as edges are deleted, but also is monotone with respect to graph minors more generally.
The purpose of this paper is to complete the set of four by introducing, and commencing the study of, a canonically chosen modification of $n - \usp{G}$ that achieves monotonicity with respect to graph minors. Thus, we study the minor-monotone floor of $n - \usp{G}=\uspc{G}$, that is, the minimum of $\uspc{H}$ among all graphs $H$ containing $G$ as a minor. We denote this minor-monotone floor by $\uspcf{G}$ and call it the \emph{spectator floor} of $G$.
{\color{black}
(This perpetuates an existing abuse of the standard notation for integer floor, as distinguished by context when it is applied to the name of a graph parameter---typically one already taking integer values---rather than to a real number.)
}

One of our main results is that, in order to find a graph $G'$ containing $G$ as a minor and such that $\uspcf{G}=\uspc{G'}$, one need only add edges to $G$. In \cref{Minor Operations and the Spectator Floor}, we will prove the following.

\begin{theorem}
\label{lem:no decontract-intro}
For every graph $G$, there is a graph $G'$ with the same number of vertices as $G$ such that $G$ is a subgraph of $G'$ and such that $\uspcf{G}=\uspc{G'}$.
\end{theorem}

Another main result is that the spectator floor is additive over connected components. We prove the following result in \cref{Disconnected Graphs}.

\begin{theorem}
\label{prop:components-intro}
For any {\color{black}disjoint union of graphs} $G=G_1\sqcup G_2$,
\[\uspcf{G}=\uspcf{G_1}+\uspcf{G_2}.\]
\end{theorem}

\cref{trees} gives some results about the spectator floor of trees. In \cref{sec:minimal}, we study graphs with small spectator floors. In particular, we determine the minor minimal graphs (whether or not parallel edges are allowed) that have spectator floor $k$, when $k=1$ and when $k=2$. In \cref{minor max graphs}, we characterize the minor maximal graphs with a given spectator floor, subject to a restriction on the number of vertices in the graph and the number of parallel edges between any pair of vertices. Finally, \cref{further questions} presents some questions for further research. Before moving on to \cref{Minor Operations and the Spectator Floor}, we continue this introduction with some additional preliminary information.

% The Strong Arnold Property, introduced by Yves Colin de Verdi\`ere, is a genericity
% condition that transforms $\M(G)$ into a graph parameter that is monotone with
% respect to the taking of graph minors, with the useful consequence that it is characterized,
% for each particular nullity, by a finite list of minor-minimal graphs that achieve that
% nullity in a real symmetric matrix with the Strong Arnold Property.
% This suggested the study of minor-monotone versions of zero forcing, and
% in \cite{Parameters} it was discovered that taking the minor-monotone floors
% of three variants of zero forcing produces three minor-monotone
% graph parameters that already known in the graph theory literature and related
% to tree width, namely
% path width, proper path width, and largeur d'arborescence.
% 
% More recently, other strong matrix parameters have been introduced, called
% the Strong Multiplicity Property (SMP) and Strong Spectral Property (SSP),
% that cause $q(G)$ also to behave well with respect to the taking of graph
% minors.
% It is thus natural to introduce and study a minor-monotone version
% of the combinatorial bound on $q(G)$.

\subsection{Definitions}
\label{definitions}

We allow a graph $G$ to have multiple edges (unless $G$ is explicitly stated
to be a simple graph) but not to have loops at its vertices. Moreover, we assume all graphs are nonempty, that is every graph has at least one vertex.
The multigraph convention for matrix entries is that each
edge of $G$ contributes additively a non-zero amount to the matrix entry,
which in particular allows the contributions from multiple edges to cancel.
The convention is also that diagonal entries are unconstrained.
Concretely, then, given a graph whose vertices are the index set
$\{1, \dots, n\}$,
the space $\mathcal{S}(G)$ of matrices conforming
to the pattern of $G$ is the set of all real symmetric $n \times n$
matrices $A = [a_{ij}]$ such that
\begin{itemize}
    \item if $i \ne j$ and there is no edge in $G$ connecting $i$ to $j$,
    then $a_{ij} = 0$,
    \item if $i \ne j$ and there is exactly one edge in $G$ connecting $i$
    to $j$, then $a_{ij} \ne 0$,
    \item if $i \ne j$ and there is more than one edge in $G$ connecting $i$
    to $j$, then $a_{ij}$ is unconstrained, and
    \item diagonal entries $a_{ii}$ are unconstrained.
\end{itemize}

\color{black}
A \emph{minor} of a graph $H$ is obtained from $H$ by a sequence of the following operations:

\begin{enumerate}
    \item deletion of an isolated vertex,
    \item deletion of an edge, denoted $H\backslash e$, or
    \item contraction of an edge $e$ with no edges in parallel with it, denoted $H/e$. (If the minor is to be a simple graph, then the edge cannot be in a triangle.)
\end{enumerate}
A minor obtained by performing exactly one of these operations is called an \emph{elementary minor}.

\color{black}

The following definitions lead up to the promised naming and explanation of $n - \usp{G}$,
together with its minor-monotone floor.
Let $G$ be a graph on $n$ vertices. We start with some standard definitions.
\begin{itemize}
    \item A \emph{walk} in $G$ from $u$ to $v$ is a sequence of
    edges $(e_1, \dots, e_k)$ from $G$ and a sequence of vertices
    $(u=v_1, v_2, \dots, v_{k}, v_{k+1}=v)$ from $G$ such that each edge $e_i$ has endpoints $v_i$ and $v_{i+1}$.
    The \emph{length} of the walk is the length $k \ge 0$ of the sequence of edges.
    \item A \emph{path} in $G$ is a walk all of whose vertices are distinct.
    The length of a path is the number of edges $k$, but the \emph{order} of a path is the number of vertices $k + 1$.
    \item A \emph{shortest path} in $G$ is a path in $G$ from $u$ to $v$ of order $k + 1$
    such that no path in $G$ from $u$ to $v$ has order $k$ or smaller.
    \item A \emph{unique shortest path} in $G$ is a shortest path $P$ in $G$ from $u$ to $v$ of order $k + 1$ such that every path in $G$ from $u$ to $v$ of order exactly $k + 1$ is identical to $P$.
    The graph $G$ may be a multigraph, but no edge in a unique shortest path can have other edges parallel to it,
    because two paths with a different sequence of edges are not considered identical, even if the sequence of vertices is the same.
\end{itemize}

\textcolor{black}{As is often done in graph theory, we will use the terms \emph{walk} and \emph{path} both as defined above (as a sequence of vertices and edges) but also as the graph or subgraph with those vertices and edges.}

We now introduce the following definitions.

\begin{itemize}
    \item A \emph{parade} in $G$ is a unique shortest path in $G$ that achieves the largest possible order
    for a unique shortest path in $G$.
    \item The \emph{parade number} of a graph $G$, denoted $\usp{G}$, is the number of vertices in some parade in $G$.
    \item The \emph{spectator number} of a graph $G$, denoted $\uspc{G}$, is the number of vertices outside some parade in $G$, hence $\uspc{G}=n-\usp{G}$.
    \item The \emph{spectator floor} of a graph $G$, denoted $\uspcf{G}$, is the minor-monotone floor of $\uspc{G}$; in other words, the minimum value of $\uspc{H}$ over the set of all graphs $H$ of which $G$ is a minor.
\end{itemize}

\subsection{Matrix-theoretic graph parameters related to \texorpdfstring{$q(G)$}{q(G)}}

%\tracy{Much of this is left over from the original problem description and still needs cleaning up.}

The graph parameters $q_M(G)$ and $q_S(G)$, introduced in \cite{SSP2017}, satisfy $q(G) \le q_M(G) \le q_S(G)$. For these variants of $q(G)$, matrices are restricted to those satisfying the Strong Multiplicity Property (SMP) or the Strong Spectral Property (SSP), respectively.
%Much as minimum rank has a complementary parameter in maximum nullity,
%one can define the \emph{maximum spectral equality} of a graph,
%which for a graph $G$ on $n$ vertices is equal to $n - q(G)$ and
%measures the maximum number of times that equality can occur in the series
%of inequalities $\lambda_1 \le \lambda_2 \le \dots \le \lambda_n$,
%taken over all matrices in $\mathcal{S}(G)$.
A consequence of \cite{SSP2020} is that the maximum SMP spectral equality, $n - q_M(G)$,
and maximum SSP spectral equality, $n - q_S(G)$,
are minor-monotone graph parameters.
%(as is the parameter $n - q_S(G)$, but
%the focus here is on multiplicity lists rather than exact spectra).
A consequence of \cite{SSP2017} is that $n - q_M(G)$ and $n - q_S(G)$ each take the sum
over components of a disconnected graph, in contrast for example
to minor-monotone matrix nullity graph parameters that tend to take
the maximum over components.
%Of the two strong eigenvalue equality parameters,
%$n - q_M(G)$ is perhaps the more relevant---considering that
%the number of eigenvalue equalities is a property
%of the multiplicity list rather than of the particular
%spectrum---and it also, of the two, provides a tighter
%minor-monotone lower bound on the maximum spectral equality
Since $q(G) \le q_M(G) \le q_S(G)$, we also have the inequalities
\[
n - q_S(G) \le n - q_M(G) \le n - q(G).
\]

The existence of a minor-monotone lower bound
invites inquiry into the minor-monotone floor of the maximum
spectral equality, which by general properties of minor-monotone
floors and ceilings shares the same lower bound:
\[
n - q_M(G) \le \lfloor n - q(G) \rfloor \le n - q(G).
\] {\color{black}(As with the notation for the spectator floor, we use $\lfloor n - q(G) \rfloor$ to denote the minor-monotone floor of $n - q(G)$, which is the minimum of $n - q(H)$ among all graphs $H$ containing $G$ as a minor.)}
%The maximum spectral equality itself certainly does not take the
%sum over components of a disconnected graph---it can be
%quite a bit larger than the sum---but it is unclear whether or
%not that might be true for its minor-monotone floor,
%which looks to be a difficult graph parameter to say anything about directly.

\medskip
\noindent
{\bf Combinatorial bounds.}
Given a graph $G$ on $n$ vertices, if the vertices $u$ and $v$
are connected by a unique shortest path on $k$ vertices, then it is straightforward
to show, for any matrix $A \in \mathcal{S}(G)$, that the
powers $I = A^0, A^1, A^2, \dots, A^{k-1}$ form a linearly independent set in the
vector space of symmetric $n \times n$ matrices, implying $q(G) > k - 1$
since the linear combination of powers in the minimal polynomial
of a matrix is zero.
This yields the result that
$q(G) \ge k$, where $k$
is the number of vertices in any unique shortest path.

This bound suggests the definition of the \emph{spectator number} given above. In particular, the spectator number is the minimum cardinality of a vertex set that is complementary, relative to the set of all vertices of $G$, to the set of vertices in a unique shortest path in $G$.
The above inequality $q(G) \ge k$,
satisfied whenever there exist $k$ vertices
forming a unique shortest path, guarantees that the spectator number is an upper bound for the maximum spectral equality $n - q(G)$,
just as the combinatorial parameter $\Z(G)$ is an upper bound for
the matrix parameter $\M(G)$.
Maximum nullity gives a bound on maximum spectral equality,
\[
\M(G) - 1 \le n - q(G),
\]
because any one eigenvalue of multiplicity $k$
induces $k - 1$ eigenvalue equalities on its own.
In a parallel way, but for entirely different and combinatorial reasons,
the quantity $\Z(G) - 1$ is a lower bound for
the spectator number. \textcolor{black}{Equivalently the spectator number plus one is an upper bound for $\Z(G)$,
because the complement of the vertices in a unique shortest path, together with one endpoint of that path, form a zero forcing set for which a zero forcing sequence exists with only one non-trivial zero forcing chain.}
Unlike zero forcing, whose computation is in general NP-hard \cite{Hardness}, the minimum
spectator number can be computed in polynomial time for any graph (see, e.g., \cref{USP-in-poly-time}).

\medskip
\noindent
{\bf Bounds for minor-monotone floors.}
The minor-monotone floor of a graph parameter
on a graph $G$ is the minimum of the parameter
over the infinite collection of graphs $H$
of which $G$ is a minor.
%, and so
%there is no a priori guarantee that the minor-monotone floor of
%a graph parameter is computable,
%let alone easily,
%even when the base parameter is easily computable.
%For zero forcing and its variants, there is a fortunate result
%that the minimum can be taken over graphs on no more vertices
%than the original graph, but no such bound appears to hold
%for the minimum USP-complement.
%It is likely, however, that there exist other methods
%to recognize and bound the number
%or type of decontractions that must be allowed when taking
%a minimum for the purposes of the minor-monotone floor,
%and thus that some algorithm, perhaps even
%a good algorithm, can be developed for computing the parameter.

%A linear time algorithm is known to exist

When a parameter $\beta(G)$ is minor-monotone, it is known that for any fixed $k$ the class of graphs with $\beta(G) \ge k$ can be recognized in polynomial time, and in fact can be recognized in linear time in the special case that the complete list of minor-minimal graphs for $\beta(G) \ge k$ is known and all of them happen to be planar graphs \cite{B93}. In \cref{sec:minimal}, we show that the minor-minimal graphs with spectator floor $1$ and $2$ are all planar.

%Further work to be done is %to examine minor-minimal graphs
%for low values of the parameter, in pursuit of characterizations
%of low values in terms of forbidden minors, and 
%to examine cases in which the parameter is or is not
%a tight bound on the related matrix parameters.

The parameter $\uspcf{G}$ extends in a natural way to multigraphs,
signed graphs up to negation, and signed multigraphs up to negation,
all of which are categories in which minor-minimal sets
are guaranteed to be finite.
The signed variant is in terms of \emph{monotone
shortest paths}; a path of length $k$ from vertex $u$ to vertex $v$ in a signed graph $G$ is a monotone shortest path
if there is no path from $u$ to $v$ of length shorter than $k$,
and if every other path from $u$ to $v$ of length $k$
has the same product of edge signs.

\section{Minor Operations and the Spectator Floor}
\label{Minor Operations and the Spectator Floor}

When calculating the spectator floor, we take the minimum of the spectator number over any graph that can be made by performing the three minor operations above \textcolor{black}{(in Section \ref{definitions})} in reverse.  The following terminology will be useful to describe the operation that reverses contraction.

\begin{definition}\label{def:decontraction}
If $G$ and $H$ are graphs and $e\in E(H)$ such that $G=H/e$, then we call the operation used to obtain $H$ from $G$ a \emph{decontraction} of a vertex.
\end{definition}

Note that there may be many ways to decontract a vertex. Therefore, decontraction is not well-defined without additional context.

The main result of this section is \cref{lem:no decontract}, which tells us that in order to calculate the spectator floor, it is not necessary to add vertices or to perform decontraction.  The only necessary operation is adding edges.  The following lemmas, \cref{lem:no-delete-vertex,lem:isolated,lem:no-isolated,lem:still-usp,lem:no-endpoints,lem:e not in pert}, are all in support of the main result.

The lemma below tells us that when taking a minor, it is not necessary to perform step 1 (deletion of an isolated vertex), and that we may order the steps so that all instances of step 3 (contraction) appear before all instances of step 2 (edge deletion). \textcolor{black}{This lemma is a standard fact. For example, it follows from the information given in \cite[Section 1.7]{Diestel2016}.}

\begin{lemma}
\label{lem:no-delete-vertex}
Let $G$ and $H$ be graphs such that $G$ is a minor of $H$. If $H$ has no isolated vertices, then there are sets of edges $C$ and $D$ such that $G\cong H/C\backslash D$.
\end{lemma}

The next lemma tells us that the presence of isolated vertices does not affect the spectator floor value.

\begin{lemma}
\label{lem:isolated}
Let $G+v$ be a graph with an isolated vertex $v$, %such that every edge incident with $v$ is a loop,
and let $G$ be the graph obtained from $G+v$ by deleting $v$. Then $\uspcf{G+v}=\uspcf{G}$.

Moreover, if $G$ is a minor of a graph $H$ and $\uspcf{G}=\uspc{H}$, then there is a graph $H^+$, with exactly one more vertex than $H$, such that $\uspcf{G+v}=\uspc{H^+}$ and $G+v$ is a minor of $H^+$.
\end{lemma}

\begin{proof}
Since $G$ is a minor of $G+v$, we have $\uspcf{G}\leq\uspcf{G+v}$.

Let $P$ be a parade in $H$. Let $H^+$ be obtained from $H$ by adding a vertex $v$ of degree $1$ %incident with exactly as many loops as $v$ in $G+v$ and
adjacent to one of the endpoints of $P$. Let $P^+$ be the path in $H^+$ obtained from $P$ by adding the vertex $v$. Since $P$ is a parade in $H$, it has $|V(H)|-\uspcf{G}$ vertices. Therefore, $P^+$ has $|V(H^+)|-\uspcf{G}$ vertices. Thus, $\uspc{H^+}\leq\uspcf{G}$. But $G$ is a minor of $H^+$. Therefore, the reverse inequality holds, and we have $\uspc{H^+}=\uspcf{G}$.

Since $G+v$ is a minor of $H^+$, we have $\uspcf{G+v}\leq\uspc{H^+}=\uspcf{G}\leq\uspcf{G+v}$. Therefore, we have equality, and the result holds.
\end{proof}

The next lemma tells us that if $H$ is a graph that realizes the spectator floor value of a graph $G$ with no isolated vertices, then $H$ also has no isolated vertices.

\begin{lemma}
\label{lem:no-isolated}
Suppose $G$ is a graph without isolated vetices, and suppose $G$ is a minor of $H$. If $\uspcf{G}=\uspc{H}$, then $H$ has no isolated vertices.
%\kevin{We can drop the requirement that $G$ has no isolated vertices if we have the result about sums over connected components. Even if not, it can still be dropped with a little work.}
\end{lemma}

\begin{proof}
Suppose otherwise for a contradiction. Let $W$ be the set of isolated vertices of $H$. It is not difficult to see that deletion of such a vertex results in a graph with spectator number reduced by $1$. (This is because $G$ is neither empty nor $K_1$. Therefore, $H\neq K_1$.) Therefore, $\uspc{H-W}=\uspc{H}-|W|$. Since no vertex of $G$ is isolated, $G$ is a minor of $H-W$. This contradicts the assumption that $\uspcf{G}=\uspc{H}$.
\end{proof}

The next definition gives us a very precise definition of what it means to contract an edge in a path.  This will be of much use in further proofs.

\begin{definition}
\label{def:P/e}
Let $P$ be a path and $e$ an edge in a graph $G$. We define $P/e$ to be the subgraph of $G/e$ obtained from $P$ by contracting $e$. More precisely, if $v$ and $w$ are the endpoints of $e$ and $v'$ is the vertex that results from contracting $e$, then:
\begin{itemize}
    \item If $P$ contains $e$, then $P/e$ is the path whose vertex set is $(V(P)-\{v,w\})\cup\{v'\}$ and whose edge set is \textcolor{black}{obtained from $E(P)-\{e\}$ by relabelling both $v$ and $w$ as $v'$}.
    \item If $P$ contains exactly one endpoint of $e$ (say $v$), then $P/e$ is \textcolor{black}{a path that is ismorphic to $P$ as a graph. The isomorphism is done by  relabelling $v$ as $v'$.}
    \item If $P$ contains both endpoints of $e$, but not $e$ itself, then $P/e$ is the graph (with exactly one cycle) whose vertex set is $(V(P)-\{v,w\})\cup\{v'\}$ and whose edge set is \textcolor{black}{obtained from $E(P)$ by relabelling both $v$ and $w$ as $v'$}.
    \item If $P$ contains neither endpoint of $e$, then $P/e=P$. 
\end{itemize}
\end{definition}

The next lemma \textcolor{black}{is straightforward and} tells us that the property of being a unique shortest path is preserved after contraction of an edge in the path.

\begin{lemma}
\label{lem:still-usp}
Let $G=H/e$.  If $P$ is a unique shortest path in $H$ and $P$ contains $e$, then $P/e$ is a unique shortest path in $G$.
\end{lemma}

%Let $G$ and $H$ be graphs such that $G=H/e$. Let $u$ and $w$ be the endpoints of $e$, and let $v$ be the vertex of $G$ formed by contracting $e$. We denote by $P/e$ the subgraph of $G$ formed from $P$. That is, $V(P/e)=V(P)-\{u,w\}\cup\{v\}$, and $E(P/e)=E(P)-\{e\}$, where the edges of $P$ incident with exactly one of $\{u,w\}$ in $H$ are incident with $v$ in $G$ instead. Note that, if $P$ contains $e$, then $P/e$ is simply a path.

The next lemma tells us that if contracting an edge in a parade increases the spectator number of the graph, then the resulting contracted parade is no longer a parade after contraction.

\begin{lemma}
\label{lem:no-endpoints}
Let $G=H/e$.  If $\uspc{H}<\uspc{G}$ and $P$ is a parade in $H$ and $P$ contains edge $e$, then $P/e$ is not a parade in $G$.
\end{lemma}
\begin{proof}
Suppose not.  Suppose $\uspc{H}<\uspc{G}$ and $P$ is a parade in $H$ and $P$ contains edge $e$ and $P/e$ is a parade in $G$.  Then
\[\usp{G}=|P/e|=|P|-1=\usp{H}-1.\]

However, since $\uspc{H}<\uspc{G}$, we also have
\[|H|-\usp{H}<|G|-\usp{G}\]
\[|G|+1-\usp{H}<|G|-\usp{G}\]
\[1+\usp{G}<\usp{H}\]
\[\usp{G}<\usp{H}-1\]
\[\usp{G}\leq \usp{H}-2.\]

This is a contradiction.
\end{proof}

The next lemma tells us that contracting an edge contained in a parade cannot increase the spectator number of the graph.

\begin{lemma}
\label{lem:e not in pert}
If $P$ is a parade in $H$ and $P$ contains both of the endpoints of edge $e$, then $\uspc{H}\geq\uspc{H/e}$.
\end{lemma}
\begin{proof}
Since $P$ is a parade in $H$, note that $P$ must contain the edge $e$. Otherwise, $P$ is not a unique shortest path.

Suppose for a contradiction that $\uspc{H}<\uspc{H/e}$. Call $H/e$ by the name $G$.  We know from \cref{lem:no-endpoints} that $P/e$ is not a parade in $G$.  But $P/e$ is still a unique shortest path by \cref{lem:still-usp}, so it must not be a longest unique shortest path anymore.

However, there must be some parade in $G$.  Let us call that $Q/e$, so that we can call the uncontracted version in $H$ by the name $Q$.

Supposing $|P|=\ell$ in $H$, then $|Q|\leq \ell$ because $P$ was a parade in $H$.  And since $P/e$ is not long enough to be a parade in $G$, $|Q/e|>\ell-1$.  Putting these together, we get
\[\ell-1<|Q/e|\leq |Q|\leq \ell.\]

This implies that
\[|Q/e|= |Q|= \ell.\]
(That is, $Q/e=Q$, and $Q$ does not contain $e$.)

Hence $\usp{G}=\usp{H}$.  However, $\uspc{H}<\uspc{G}$ implies $\usp{G}\leq \usp{H}-2$.  This is a contradiction.
\end{proof}

%It is not difficult to see that loops have no effect on the length of the shortest path between any pair of vertices. This leads to the following.

Finally, we have the main result of this section, which tells us that for any graph $G$, we can always find a supergraph $G'$ with the same number of vertices that realizes the spectator floor value of $G$.  This is one of the major results of the paper and supports further results in other sections.  We now prove \cref{lem:no decontract-intro} restated below.

\begin{theorem}
\label{lem:no decontract}
For every graph $G$, there is a graph $G'$ with the same number of vertices as $G$ such that $G$ is a subgraph of $G'$ and such that $\uspcf{G}=\uspc{G'}$.
\end{theorem}

\begin{proof}
By \cref{lem:isolated}, it suffices to consider the case where $G$ has no isolated vertices. Let $H$ be a graph such that $G$ is a minor of $H$ and such that $\uspcf{G}=\uspc{H}$. By \cref{lem:no-isolated}, $H$ has no isolated vertices. Therefore, by \cref{lem:no-delete-vertex}, there are sets $C,D$ of edges such that $G\cong H/C\backslash D$, without the necessity to delete vertices. Now, suppose that $H$ is such that $|C|$ is minimal among all such graphs. Suppose for a contradiction that $C\neq\emptyset$.

For some edge $e\in C$, consider the graph $H/e$. Let $u$ and $v$ be the endpoints of $e$, and let $v'$ be the vertex that results from contracting $e$. By minimality of $C$, we have $\uspc{H}<\uspc{H/e}$. (Otherwise, the set of edges that need to be contracted from $H/e$ to obtain $G$ is smaller than $C$.) Let $P$ be a parade of $H$. By \cref{lem:e not in pert},  $P$ does not contain both endpoints of $e$. Therefore, \textcolor{black}{by Definition \ref{def:P/e}, $P/e$ is a path in $H/e$ with the same length as $P$}. However, since $H/e$ has one fewer vertex than $H$, the fact that $\uspc{H}<\uspc{H/e}$ implies that $\usp{H/e}\leq\usp{H}-2$. This implies that \textcolor{black}{$P/e$} is not a parade in $H/e$.

Since {\color{black}$P$ is a shortest path in $H$, it is clear that $P/e$ is also a shortest path in $H/e$. However, since $P/e$} is longer than the parades of $H/e$, there must be at least one path \textcolor{black}{$Q\neq P/e$} in $H/e$ with the same length and endpoints as \textcolor{black}{$P/e$}. Let $\{P_1,\ldots,P_t\}$ be the set of all such paths $Q$. {\color{black}Note that each $P_i$ is a shortest path in $H/e$.} Since $P$ is a parade in $H$, \textcolor{black}{there is no other path in $H$ with the same endpoints and length as $P$}. Therefore, for each of these paths $P_i$ for $1\leq i\leq t$, there is a path $P_i'$ in $H$, containing $e$, such that $P_i'/e=P_i$. \textcolor{black}{Note that this implies that $v'$ is a vertex on each of these paths $P_i$.}

\begin{claim}
\label{3vertices}
$|V(P)|=\usp{H}\geq3$
\end{claim}

\begin{subproof}
Contracting an edge cannot result in an empty graph. Therefore, $|V(H/e)|\geq1$, implying that $\usp{H/e}\geq1$. Thus, since $\usp{H}\geq\usp{H/e}+2$, we have $\usp{H}\geq3$.
\end{subproof}

We will add an edge $f$ to $H/e$ to result in a graph $F$ such that $\uspc{F}\leq\uspc{H}$, contradicting the minimality of $H$. Let $x=v_0,e_1,v_1,\ldots,e_r,v_r=y$ be the succession of vertices and edges of {\color{black}$P/e$.

\begin{claim}
\label{same_vertex}
Suppose $x=w_0,f_1,w_1,\ldots,f_r,w_r=y$ is the succession of vertices and edges of $P_i$. If $w_j$ is a vertex in $P/e$, then $w_j=v_j$.
\end{claim}

\begin{proof}
    \textcolor{black}{First note that, if $w_j=v_{\alpha}$ and $w_k=v_{\beta}$ are two vertices on both $P/e$ and $P_i$, then $k-j=\beta-\alpha$. Otherwise, the union of $P/e$ and $P_i$ contains a path from $x$ to $y$ that is shorter than both $P/e$ and $P_i$, which cannot be true because $P/e$ and $P_i$ are shortest paths. Since $v_0=w_0$, the claim follows.}
\end{proof}

The previous claim shows that each time $P/e$ and $P_i$ diverge from each other, the resulting pair of internally disjoint paths have the same length. We claim that this only happens once, as explained in Claim \ref{1divergence} below. 

For each integer $i$ with $1\leq i\leq t$, let $m(i)$ be the smallest nonnegative integer such that $v_{m(i)}$ is a vertex of $P_i$ but $e_{m(i)+1}$ is not an edge of $P_i$, and let $n(i)$ be the largest positive integer such that $e_{n(i)}$ is not an edge of $P_i$ but $v_{n(i)}$ is a vertex of $P_i$. (These are both well-defined since $P/e$ and $P_i$ have the same endpoints.)

\begin{claim}
\label{1divergence}

For each integer $i$ with $1\leq i\leq t$, the subpaths of $P_i$ and $P/e=P$ from $v_{m(i)}$ to $v_{n(i)}$ are internally disjoint.
\end{claim}

\begin{subproof}
Suppose for a contradiction that there is a positive integer $p$ such that $m(i)<p<n(i)$ and  such that $P_i$ contains the vertex $v_p$. Recall that $P_i=P'_i/e$, where $P_i'$ is a path in $H$, and $e$ is an edge in $P_i'$. If $v'$ is a vertex in $P/e$, then let $v$ be the endpoint of $e$ that is in $P$, and let $u$ be the other endpoint of $e$. Then, regardless of whether $v'$ is a vertex in $P/e$, the fact that $e$ is not an edge of $P$ implies that $e$ is either on the subpath of $P'_i$ from $v_{m(i)}$ to $v_p$ or on the subpath of $P'_i$ from $v_p$ to $v_{n(i)}$.

We may assume without loss of generality that $e$ is on the subpath of $P'_i$ from $v_{m(i)}$ to $v_p$. (Otherwise, reverse the order of the vertices in $P/e$.) This means that $v'$ is on the subpath of $P_i$ from $v_{m(i)}$ to $v_p$ (including the possibility that $v'=v_{m(i)}$ or $v'=v_p$). 

Let $R$ be the path consisting of the subpath of $P/e$ from $v_0$ to $v_p$ and the subpath of $P_i$ from $v_p$ to $v_r$. Then $R$ is a path from $v_0$ to $v_r$ that is either a path in $H$ or obtained from a path in $H$ by relabeling the vertex $v$ as $v'$. In either case, this implies that $R$ must be longer than $P$. Thus, the subpath of $P_i$ from $v_p$ to $v_r$ is longer than the subpath of $P/e$ from $v_p$ to $v_r$, but this contradicts Claim \ref{same_vertex}. 
\end{subproof}

By} \cref{1divergence}, there is exactly one subpath of each $P_i'$ whose intersection with {\color{black} $P/e$} is the endpoints of the subpath. Let $p(i)$ be the length (number of edges) of the subpath of $P_i$ from $v_{m(i)}$ to $v'$, and let $q(i)$ be the length of the subpath of $P_i$ from $v'$ to $v_{n(i)}$. Let the sequence of vertices of $P_i$ be $x=v_0,v_1,\ldots,v_{m(i)},u_1,\ldots,u_{p(i)-1},u_{p(i)}=v'=w_{q(i)},w_{q(i)-1},\ldots,w_1,v_{n(i)},\ldots,v_r=y$. (See \cref{fig:PandP'}; note that it is possible $v_{m(i)}=v'$ or $v_{n(i)}=v'$, in which case $p(i)=0$ or $q(i)=0$, respectively.)

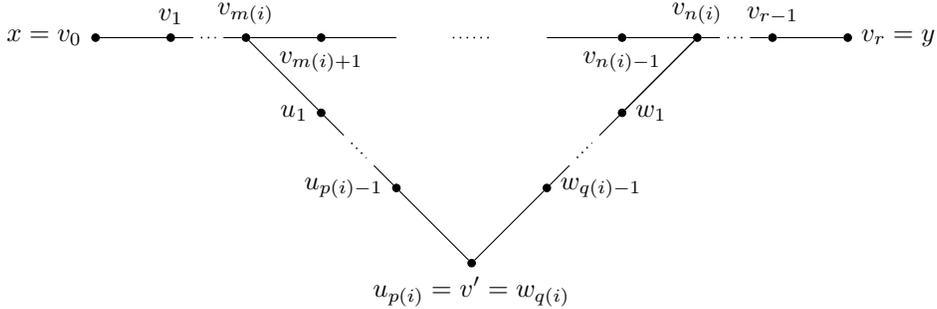
\begin{figure}[!htbp]
\[\begin{tikzpicture}[x=1cm, y=1cm]
\vertex[fill,inner sep=1pt,minimum size=1pt] (v_0) at (-5,3)[label=left:$x\equal v_0$] {};
\vertex[fill,inner sep=1pt,minimum size=1pt] (v_1) at (-4,3)[label=above:$v_1$]{};
\vertex[fill,inner sep=1pt,minimum size=1pt] (vmi) at (-3,3)[label=above:$v_{m(i)}$]{};
\vertex[fill,inner sep=1pt,minimum size=1pt] (vmi+1) at (-2,3)[label=below:$v_{m(i)+1}$]{};
\vertex[fill,inner sep=1pt,minimum size=1pt] (vni-1) at (2,3)[label=below:$v_{n(i)-1}$]{};
\vertex[fill,inner sep=1pt,minimum size=1pt] (vni) at (3,3)[label=above:$v_{n(i)}$]{};
\vertex[fill,inner sep=1pt,minimum size=1pt] (vr-1) at (4,3)[label=above:$v_{r-1}$]{};
\vertex[fill,inner sep=1pt,minimum size=1pt] (v_r) at (5,3)[label=right:$v_r\equal y$] {};
\vertex[fill,inner sep=1pt,minimum size=1pt] (u_1) at (-2,2)[label=left:$u_1$] {};
\vertex[fill,inner sep=1pt,minimum size=1pt] (up-1) at (-1,1)[label=left:$u_{p(i)-1}$] {};
\vertex[fill,inner sep=1pt,minimum size=1pt] (v') at (0,0)[label=below:$u_{p(i)}\equal v'\equal w_{q(i)}$] {};
\vertex[fill,inner sep=1pt,minimum size=1pt] (wq-1) at (1,1)[label=right:$w_{q(i)-1}$] {};
\vertex[fill,inner sep=1pt,minimum size=1pt] (w_1) at (2,2)[label=right:$w_1$] {};

\draw(v_0.center) to (v_1.center) to (-3.7,3); \draw[dotted,semithick] (-3.6,3) to (-3.4,3); \draw (-3.3,3) to (vmi) to (vmi+1) to (-1,3); 
\draw[dotted,semithick] (-.25,3) to (.25,3); \draw (1,3) to (vni-1) to (vni) to (3.3,3); \draw[dotted,semithick] (3.6,3) to (3.4,3); \draw (3.7,3) to (vr-1) to (v_r);
\draw(vni) to (w_1);\draw (vmi) to (u_1) to (-1.7,1.7);
\draw[dotted,semithick] (-1.6,1.6) to (-1.4,1.4);
\draw (-1.3,1.3) to (up-1) to (v') to (wq-1) to (1.3,1.3); \draw[dotted,semithick] (1.4,1.4) to (1.6,1.6); \draw (1.7,1.7) to (w_1) to (vni); 

\end{tikzpicture}\]
\caption{The paths $P$ and $P_i$}
\label{fig:PandP'}
\end{figure}

Based on this labeling of the vertices, we see that $|{\color{black}P/e}|=|P_i|=1+m(i)+p(i)-1+q(i)+r-(n(i)-1)=m(i)+p(i)+q(i)+r-n(i)+1$. Recall that $P_i$ and $P_j$ have the same length, for $i,j\in\{1,\ldots,t\}$. Therefore, for all $i,j\in\{1,\ldots,t\}$, we have \[m(i)+p(i)+q(i)-n(i)=m(j)+p(j)+q(j)-n(j).\]

\begin{claim}
For all $i,j\in\{1,\ldots,t\}$, we have $m(i)+p(i)=m(j)+p(j)$ and $q(i)-n(i)=q(j)-n(j)$.
\end{claim}

\begin{subproof}
Suppose for a contradiction that $m(i)+p(i)<m(j)+p(j)$. Since $m(i)+p(i)+q(i)-n(i)=m(j)+p(j)+q(j)-n(j)$, we have $q(i)-n(i)>q(j)-n(j)$. Consider the {\color{black}path in} $H/e$ consisting of the subpath of $P_i$ from $x$ to $v'$ and the subpath of $P_j$ from $v'$ to $y$. This {\color{black}is} a path from $x$ to $y$ whose number of vertices is $m(i)+p(i)+q(j)+r-n(j)+1<m(j)+p(j)+q(j)+r-n(j)+1=|P|$. {\color{black} This path} implies the existence of a path in $H$ from $x$ to $y$ whose length is no longer than that of $P$, a contradiction.

Thus, we have proved the first equality of the claim. The second equality follows from the first equality and the fact that $m(i)+p(i)+q(i)-n(i)=m(j)+p(j)+q(j)-n(j)$.
\end{subproof}

\begin{claim}
\label{mlessn}
For all $i,j\in\{1,\ldots,t\}$, we have $m(i)<n(j)$.
\end{claim}

\begin{subproof}
Suppose for a contradiction that $m(i)\geq n(j)$. It follows from \cref{1divergence} that the number of vertices on the subpaths of {\color{black}$P/e$} and $P_i$ from $v_{m(i)}$ to $v_{n(i)}$ must be equal. Similarly, the number of vertices on the subpaths of {\color{black}$P/e$} and $P_j$ from $v_{m(j)}$ to $v_{n(j)}$ must be equal. The number of vertices of the subpath of {\color{black}$P/e$} from $v_{m(i)}$ to $v_{n(i)}$ is $n(i)-m(i)+1$, and the number of vertices of the subpath of {\color{black}$P/e$} from $v_{m(j)}$ to $v_{n(j)}$ is $n(j)-m(j)+1$. The number of vertices of the subpath of $P_i$ from $v_{m(i)}$ to $v_{n(i)}$ is $p(i)+q(i)+1$, and number of vertices of the subpath of $P_j$ from $v_{m(j)}$ to $v_{n(j)}$ is $p(j)+q(j)+1$. Thus, we have $n(i)-m(i)=p(i)+q(i)$ and $n(j)-m(j)=p(j)+q(j)$.

Consider the {\color{black} path $\widehat{P}$ in} $H/e$ consisting of the subpath of $P_j$ from $x$ to $v'$ and the subpath of $P_i$ from $v'$ to $y$. This {\color{black} path has} $m(j)+p(j)+q(i)+r-n(i)+1$ vertices. Therefore, {\color{black}$\widehat{P}$ is a path} from $x$ to $y$ such that $|\widehat{P}|\leq m(j)+p(j)+q(i)+r-n(i)+1$. This path must be at least as long as $P$. Therefore, $r+1\leq m(j)+p(j)+q(i)+r-n(i)+1$, implying that $n(i)-m(j)\leq p(j)+q(i)$.

On the other hand, $n(i)-m(j)\geq n(i)-m(i)+n(j)-m(j)=p(i)+q(i)+p(j)+q(j)$. This leads to a contradiction unless $p(i)=q(j)=0$.

Thus, $v_{m(i)}=v'=v_{n(j)}$ is a vertex of {\color{black}$P/e$}. This implies $m(i)=n(j)$. Let $u$ be the endpoint of $e$ not on $P$. Consider the {\color{black}path in} $H$ consisting of the subpath of $P_j'$ from $x$ to $u$ and the subpath of $P_i$ from $u$ to $y$. This subgraph contains a path from $x$ to $y$ whose number of vertices is at most $m(j)+p(j)+q(i)+r-n(i)+1=m(i)+p(i)+q(j)+r-n(j)+1=r+1=|P|$, a contradiction.
\end{subproof}

Let $m=\max\{m(i):1\leq i\leq t\}$ and $n=\min\{n(i):1\leq i\leq t\}$. By \cref{mlessn}, we have $m<n$. Recall from \cref{3vertices} that $|P|\geq3${\color{black}, implying that $|P/e|\geq3$}. Therefore, either $n\geq2$ or $m+2\leq r$. (Otherwise, $m<n<2$, which implies $m=0$. If $m=0$ and $m+2>r$, then $r<2$, implying $|{\color{black} P/e}|=r+1<3$.) Without loss of generality, we assume $m+2\leq r$, reversing the order of the vertices and edges of ${\color{black} P/e}$ if necessary. Thus, $v_{m+2}$ is a vertex of {\color{black}$P/e$}.

Construct the graph $F$ from $H/e$ by adding the edge $f$ joining $v_m$ and $v_{m+2}$. We claim that the resulting path $P_F$ whose sequence of  vertices is $v_0,\ldots,v_m,v_{m+2},\ldots,v_r$ is the unique shortest path between $x$ and $y$ in $F$. Suppose otherwise for a contradiction. Then $F$ contains a path $P_F'\neq P_F$ from $x$ to $y$ with $|P_F'|\leq|P_F|=|P|-1$.

Suppose $P_F'$ is a path in $H/e$. Then $H$ contains a path $Q$ from $x$ to $y$ with at most $|P|$ vertices such that {\color{black}$P_F'=Q/e$}. Since $P$ is the unique shortest path from $x$ to $y$ in $H$, we must have $Q=P$. Thus, {\color{black}$P_F'=Q/e=P/e$, which has the same length as $P$} since $P$ does not contain $e$.  This contradicts the assumption that $|P_F'|\leq|P|-1$. Therefore, $P_F'$ is not a path in $H/e$. Thus, $P_F'$ contains the edge $f$.

Let $P_F''$ be the path obtained from $P_F'$ by replacing edge $f$ with the subpath of {\color{black}$P/e$} from $v_m$ to $v_{m+2}$. Then $|P_F''|=|P_F'|+1\leq|P|$. Since $P_F''$ is a path in $H/e$, this implies that $P_F''=P_i$ for some $i\in\{1,\ldots,t\}$. Since $v_m$ is a vertex of $P_i$, we must have $m(i)=m$. Since $v_{m+2}$ is a vertex of $P_i$, we must have $n(i)\in\{m+1,m+2\}$. But then $P_F'=P_F$, a contradiction.

Thus, we have shown that $P_F$ is the unique shortest path between $x$ and $y$ in $F$. Therefore, $\usp{F}\geq|V(P)|-1=\usp{H}-1$. Thus, $\uspc{F}=|V(H/e)|-\usp{F}=|V(H)|-1-\usp{F}\leq|V(H)|-1-\usp{H}+1=\uspc{H}$, contradicting the minimality of $H$. 
\end{proof}

\section{Disconnected Graphs}
\label{Disconnected Graphs}

In this section, we establish that the spectator floor is additive over disconnected graph components.  We proved \cref{prop:components} independently of \cref{lem:no decontract} and have decided to include both proofs in their entirety. However, it is interesting to note that either result could be used to prove the other.  We will now prove \cref{prop:components-intro} restated below.

%  It is interesting to note that the proof of this result does not depend on \cref{lem:no decontract}. However, 

%\kevin{I want to say ``We chose to prove \cref{lem:no decontract} and \cref{prop:components} independently because...'' I know we discussed this at some point, but I don't know if we have a good reason. It might be best to switch Sections 2 and 3 and use \cref{prop:components} to prove \cref{lem:no decontract}. That way, we can delete \cref{lem:isolated} and \cref{lem:no-isolated}.}

\begin{theorem}
\label{prop:components}
For any {\color{black}disjoint union of graphs} $G=G_1\sqcup G_2$,
\[\uspcf{G}=\uspcf{G_1}+\uspcf{G_2}.\]
\end{theorem}

\begin{proof}
Let $G=G_1\sqcup G_2$, where $G$ contains no edges connecting $G_1$ to $G_2$.  The statement is trivially true when either $G_1$ or $G_2$ is empty, so assume that neither $G_1$ nor $G_2$ is empty.

For $i=1,2$, let $F_i$ be a graph that realizes $\uspcf{G_i}$; that is, $F_i$ is such that $G_i$ is a minor of $F_i$, and $F_i$ has a unique shortest path $P_i$ with $|F_i|-\uspcf{G_i}$ vertices.

Let $F$ be $F_1\sqcup F_2$ with an edge $e$ added to make a path $P$ that connects $P_1$ and $P_2$ end-to-end.  Then $G$ is a minor of $F$, and $P$ is a unique shortest path in $F$ with
\[|F_1|-\uspcf{G_1}+|F_2|-\uspcf{G_2}=|F|-\uspcf{G_1}-\uspcf{G_2}\]
vertices.  Hence $\uspc{F}\leq \uspcf{G_1}+\uspcf{G_2}$ and so $\uspcf{G}\leq \uspcf{G_1}+\uspcf{G_2}$.

Now, contrary to the statement of the result, suppose that $\uspcf{G}< \uspcf{G_1}+\uspcf{G_2}$.  Then there is some graph $H$, of which $G$ is a minor, such that $\uspc{H}=\uspcf{G}<\uspcf{G_1}+\uspcf{G_2}$.  Since $G$ is a minor of $H$, we can build $H$ from $G$ by a sequence of decontracting vertices, adding edges, and/or adding isolated vertices.  We will now construct a partition of $H$ into three \textcolor{black}{parts}: a subgraph $H_1$, a subgraph $H_2$, and a set of ``bridge'' edges $B$.

Let us start by defining $H_1$ as $G_1$ and $H_2$ as $G_2$.  Next, for each time that a vertex is decontracted while transforming $G$ into $H$, any vertices or edges that are doubled by the decontraction will remain in the same set that they were in before.  Whenever a new edge is added, if it connects two vertices in the same $H_i$, that edge will be added to $H_i$.  If the new edge connects a vertex from $H_1$ to a vertex from $H_2$, the edge will be added to the bridge edge set $B$.  Whenever a new isolated vertex is added, it will be randomly placed into either $H_1$ or $H_2$.

At the end of this process, we have that all the vertices of $H$ are in either $H_1$ or $H_2$, all the edges of $H$ are in $H_1$, $H_2$, or $B$, and no vertex or edge is in more than one of these.  Furthermore, $G_1$ is a minor of $H_1$ and $G_2$ is a minor of $H_2$.

Let $P$ be a path in $H$ that contains $\usp{H}$ vertices.  $P$ cannot be entirely within either $H_1$ or $H_2$ for the following reasons.  Suppose, without loss of generality, that $P$ is a subgraph of $H_1$.  Since $G$ is a minor of $H_1\cup G_2$, that means $\uspcf{G}\leq \uspcf{H_1\cup G_2}$.  Since $H_1$ and $G_2$ are disjoint, by the same arguments used for $G,G_1,G_2$ previously, we can conclude that $\uspcf{H_1\cup G_2}\leq \uspcf{H_1}+\uspcf{G_2}$.  Hence, $\uspcf{G}\leq \uspcf{H_1}+\uspcf{G_2}$.  This implies that
\begin{align*}
    \uspcf{G}\leq \uspc{H_1}+\uspcf{G_2} &=|H_1|-\usp{H_1}+\uspcf{G_2}\\
    &=|H_1|-\usp{H}+\uspcf{G_2}\\
    &=|H_1|+|H_2|-\usp{H}+\uspcf{G_2}-|H_2|\\
    &=|H|-\usp{H}+\uspcf{G_2}-|H_2|\\
    &=\uspc{H}+\uspcf{G_2}-|H_2|.
\end{align*}

By definition, $\uspcf{G}=\uspc{H}$, so the above inequality simplifies to $0\leq \uspcf{G_2}-|H_2|$.  Furthermore, $\uspcf{G_2}\leq \uspc{G_2}=|G_2|-\usp{G_2}$.  Hence, $0\leq |G_2|-\usp{G_2}-|H_2|$.  However, since $G_2$ is a minor of $H_2$, the quantity $|G_2|-|H_2|\leq 0$.  Hence, we have $0\leq -\usp{G_2}$, which implies that $G_2$ is an empty graph, in contradiction to our assumption that neither $G_1$ nor $G_2$ is empty.

Returning to our discussion of the graph $H$, we now know that the path $P$ with $\usp{H}$ vertices must contain vertices from both $H_1$ and $H_2$, and so $P$ must contain at least one edge in $B$. We will now show that this leads to a contradiction as well.

Suppose that $P$ contains $m$ edges in $B$, where $m\geq 1$.  Assume without loss of generality that $P$ has at least one end in $H_1$.  Then we can break $P$ up into disjoint subpaths $Q_1,Q_2,...,Q_j$ in $H_1$ and $R_1,R_2,...,R_k$ in $H_2$, such that $P$ consists of $Q_1$ followed by $R_1$ followed by $Q_2$ followed by $R_2$, etc., with these subpaths linked together by edges in the bridge $B$.  (See \cref{disconnected_diagram}.) Note that if $m$ is even, then $j=k+1$ and $m=2k$, whereas if $m$ is odd, then $j=k$ and $m=2k-1$.

\begin{figure}[!htbp]
\[\begin{tikzpicture}[x=1cm, y=1cm]

\fill[gray!20,rounded corners] (-5,2) rectangle (-1,6);
\fill[gray!20] (-5,1) rectangle (-1,3);

\fill[gray!20,rounded corners] (5,2) rectangle (1,6);
\fill[gray!20] (5,1) rectangle (1,3);

\shade[
    left color = gray!20,
    right color = white,
    shading angle = 0
] (5,0.5) rectangle (1,1);

\shade[
    left color = gray!20,
    right color = white,
    shading angle = 0
] (-5,0.5) rectangle (-1,1);

\draw[rounded corners] (-1,1) -- (-1,6) -- (-5,6) -- (-5,1); %H1
\draw[rounded corners] (1,1) -- (1,6) -- (5,6) -- (5,1); %H2
\draw[dotted,semithick] (1,1) -- (1,0.5);
\draw[dotted,semithick] (5,1) -- (5,0.5);
\draw[dotted,semithick] (-1,1) -- (-1,0.5);
\draw[dotted,semithick] (-5,1) -- (-5,0.5);

\node at (-3.5,6)[label=above:$H_1$] {};
\node at (3.5,6)[label=above:$H_2$] {};
\node at (0,6)[label=above:$B$] {};

\vertex[fill,inner sep=1pt,minimum size=1pt] (a1) at (-3,5.5)[label=below:$a_1$]{};
\vertex[fill,inner sep=1pt,minimum size=1pt] (b1) at (-1.5,5.5)[label=below left:$b_1$]{};
\draw (a1) to (-2.5,5.5);
\draw[dotted, semithick] (-2.5,5.5) to (-2,5.5);
\draw (-2,5.5) to (b1);

\node at (-4,5.5) {$Q_1$};
\node at (-4,4) {$Q_2$};
\node at (-4,2) {$Q_3$};

\node at (4,5) {$R_1$};
\node at (4,3) {$R_2$};
\node at (4,1) {$R_3$};

\vertex[fill,inner sep=1pt,minimum size=1pt] (c1) at (1.5,5.5)[label=below right:$c_1$]{};
\vertex[fill,inner sep=1pt,minimum size=1pt] (d1) at (1.5,4.5)[label=below right:$d_1$]{};
\draw (c1) to (2.5,5.5);
\draw (2.5,4.5) to (d1);
\draw (2.5,5.5) arc (90:45:0.5);
\draw (2.5,4.5) arc (-90:-45:0.5);
\draw[dotted,semithick] (2.5,5.5) arc (90:-90:0.5);

\vertex[fill,inner sep=1pt,minimum size=1pt] (a2) at (-1.5,4.5)[label=below left:$a_2$]{};
\vertex[fill,inner sep=1pt,minimum size=1pt] (b2) at (-1.5,3.5)[label=below left:$b_2$]{};
\draw (a2) to (-2.5,4.5);
\draw (-2.5,3.5) to (b2);
\draw (-2.5,4.5) arc (90:135:0.5);
\draw (-2.5,3.5) arc (-90:-135:0.5);
\draw[dotted,semithick] (-2.5,4.5) arc (90:270:0.5);

\draw (b1) to[in=150,out=30] (c1);
\draw (a2) to[in=150,out=30] (d1);

\vertex[fill,inner sep=1pt,minimum size=1pt] (c2) at (1.5,3.5)[label=below right:$c_2$]{};
\vertex[fill,inner sep=1pt,minimum size=1pt] (d2) at (1.5,2.5)[label=below right:$d_2$]{};
\draw (c2) to (2.5,3.5);
\draw (2.5,2.5) to (d2);
\draw (2.5,3.5) arc (90:45:0.5);
\draw (2.5,2.5) arc (-90:-45:0.5);
\draw[dotted,semithick] (2.5,3.5) arc (90:-90:0.5);

\vertex[fill,inner sep=1pt,minimum size=1pt] (a3) at (-1.5,2.5)[label=below left:$a_3$]{};
\vertex[fill,inner sep=1pt,minimum size=1pt] (b3) at (-1.5,1.5)[label=below left:$b_3$]{};
\draw (a3) to (-2.5,2.5);
\draw (-2.5,1.5) to (b3);
\draw (-2.5,2.5) arc (90:135:0.5);
\draw (-2.5,1.5) arc (-90:-135:0.5);
\draw[dotted,semithick] (-2.5,2.5) arc (90:270:0.5);

\draw (b2) to[in=150,out=30] (c2);
\draw (a3) to[in=150,out=30] (d2);

\vertex[fill,inner sep=1pt,minimum size=1pt] (c3) at (1.5,1.5)[label=below right:$c_3$]{};
\draw (b3) to[in=150,out=30] (c3);
\draw (c3) to (2.5,1.5);
\draw (2.5,1.5) arc (90:45:0.5);
\draw[dotted,semithick] (2.5,1.5) arc (90:0:0.5);

\draw[dashed] (b1) to[out=-45,in=45] (a2);
\draw[dashed] (b2) to[out=-45,in=45] (a3);
\draw[dashed] (d1) to[out=-135,in=135] (c2);
\draw[dashed] (d2) to[out=-135,in=135] (c3);

\end{tikzpicture}\]
\caption{A generalized diagram of the subgraph of $H$ induced by the path $P$.  Dashed lines indicate edges that will be added to form the graphs $H_1'$, $H'$, $H_2''$, and $H''$.}\label{disconnected_diagram}
\end{figure}
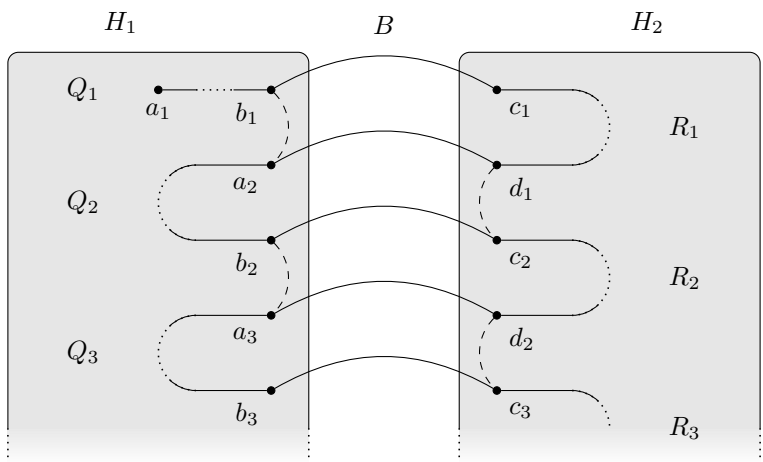

Furthermore, let us label the first vertex of each $Q_i$ by the name $a_i$, and the last vertex of each $Q_i$ by $b_i$, the first vertex of each $R_i$ by $c_i$, and the last vertex of each $R_i$ by $d_i$.  Hence, each vertex $b_i$ is connected to the vertex $c_i$ by an edge in $B$, except possibly for $b_j$, and each vertex $d_i$ is connected to the vertex $a_{i+1}$ by an edge in $B$, except possibly for $d_k$.

Let $p$ stand for the number of vertices in $P$, $q$ the total number of vertices in $Q_1,Q_2,...,Q_j$, and $r$ the total number of vertices in $R_1,R_2,...,R_k$.  Hence $p=q+r$.  By assumption, $\uspc{H}=|H|-p<\uspcf{G_1}+\uspcf{G_2}$.  Since $|H|-p=|H_1|-q+|H_2|-r$, it must be the case that either $|H_1|-q<\uspcf{G_1}$ or $|H_2|-r<\uspcf{G_2}$ (since, if both of these were false, it would contradict the original inequality).  However, we will now show that neither of these is true.

Observe that since $P$ is a unique shortest path in $H$, there cannot be an edge in $H_1$ connecting any $b_i$ to $a_{i+1}$, because that would form a shorter path.  Likewise, there cannot be an edge in $H_2$ connecting any $d_i$ to $c_{i+1}$.

Let $H_1'$ (respectively, $H'$) be formed by taking $H_1$ (respectively, $H$) and adding an edge from each $b_i$ to $a_{i+1}$, except for $b_j$, the last one.  (Note that if $j=1$, then $H_1'=H_1$.)  These edges are indicated by dashed lines in \cref{disconnected_diagram}.  Since $P$ was a unique shortest path in $H$ and none of these new edges were present previously, we now have a shorter unique path connecting the endpoints of $P$.

Let us call the portion of this new unique shortest path that lies in $H_1'$ by the name $P'$.  (Unless $j=1$, then let $P'$ be the portion of $P$ in $H_1$.)  Since any subpath of a unique shortest path is also a unique shortest path, $P'$ is a unique shortest path with $q$ vertices contained in $H_1'$, hence $\uspc{H_1'}\leq |H_1'|-q=|H_1|-q$.  Furthermore, since $G_1$ is a minor of $H_1'$, we have $\uspcf{G_1}\leq \uspc{H_1'}\leq |H_1|-q$.

Now, let $H_2''$ (respectively, $H''$) be formed by taking $H_2$ (respectively, $H$) and adding an edge from each $d_i$ to $c_{i+1}$, except for $d_k$, the last one.  (Note that if $k=1$, then $H_2''=H_2$.)  These edges are indicated by dashed lines in \cref{disconnected_diagram}.  Since $P$ was a unique shortest path in $H$ and none of these new edges were present previously, we now have a shorter unique path connecting the endpoints of $P$.

Let us call the portion of this new unique shortest path that lies in $H_2''$ by the name $P''$.  (Unless $k=1$, then let $P''$ be the portion of $P$ in $H_2$.)  Since any subpath of a unique shortest path is also a unique shortest path, $P''$ is a unique shortest path with $r$ vertices contained in $H_2''$, hence $\uspc{H_2''}\leq |H_2''|-r=|H_2|-r$.  Furthermore, since $G_2$ is a minor of $H_2''$, we have $\uspcf{G_2}\leq \uspc{H_2''}\leq |H_2|-r$.
\end{proof}

The following is a direct consequence of the last proposition.
\begin{cor}
\label{cor:components}
For any graph $G$ that consists of connected components $G_1, G_2, ..., G_n$,
\[\uspcf{G}=\uspcf{G_1}+\uspcf{G_2}+\cdots+\uspcf{G_n}.\]
\end{cor}

\section{Trees}
\label{trees}
This section contains results that apply to trees, with the exception of \cref{diam.bound} and \cref{lem:diameter}, which offer bounds that apply to all graphs, and were inspired by the consideration of trees.

We begin by noting that since $\usp{G}$ is the number of vertices in the longest unique shortest path and $\diam(G)$ is the length of the longest shortest path in $G$, which may or may not be unique, we have the following:
\begin{obs}\label{diam.bound}
For any graph $G$, $\usp{G} \leq \diam(G) + 1$ and so $\uspc{G} \geq |G| - \diam(G) - 1$.  %\craig{and so $\uspc{G} \geq n - \diam(G) - 1$}
\end{obs}

Next, we parlay \cref{diam.bound} into an exact formula for the spectator floor of a tree.

\begin{theorem}\label{tree.uspcf}
For a tree $T$, $\uspcf{T}=|T|-\diam(T)-1$.
\end{theorem}
\begin{proof}
Let $T$ be a tree.  Since all paths in a tree are unique, $\usp{T}=\diam(T)+1$, and so $\uspc{T}=|T|-\diam(T)-1$.

Suppose that $\uspcf{T}<|T|-\diam(T)-1$.  Then there exists some graph $G$ which has $T$ as a minor, such that $\uspcf{T}=\uspc{G}<|T|-\diam(T)-1$.

% In general, for any graph $G$, $\usp{G}\leq \diam(G)+1$, because $\usp{G}$ is the number of vertices in the longest unique shortest path in the graph, whereas $\diam(G)+1$ is the number of vertices in the longest shortest path in the graph, hence $\diam(G)+1$ is a maximum over a superset of the paths that $\usp{G}$ is maximizing over.

% \alathea{The above paragraph is awkwardly worded---fix this later}.

The bound $\usp{G} \leq \diam(G)+1$ implies that $|G|-\diam(G)-1\leq \uspc{G}$.  Then we have
\[|G|-\diam(G)-1<|T|-\diam(T)-1,\]
and rearranging this, we get
\[|G|-|T|<\diam(G)-\diam(T).\]
The quantity on the left, $|G|-|T|$, is the number of decontractions %\craig{We need to define ``decontraction'' in the introduction section.} 
performed to transform $T$ into $G$.  Note that one decontraction can increase the diameter of a graph by at most 1.  Furthermore, adding an edge cannot increase the diameter of a graph.  Therefore, the quantity $\diam(G)-\diam(T)$ must be less than or equal to the number of decontractions performed to transform $T$ into $G$.  This is a contradiction.
\end{proof}

The next theorem is important in establishing \cref{cor:diam-paths}, which tells us the conditions under which a tree is minor minimal for a given value of the spectator floor.  First we need two definitions.

A \emph{diametric path} of a graph is a path whose length is the diameter. The endpoints of such a path are called a \emph{diametric pair} of vertices.

\begin{theorem}
For any tree $T$, there exists
a non-empty elementary minor of $T$ with the same spectator floor as $T$ if and only if there exists
an edge $e$ of $T$ such that contracting $e$ reduces the diameter of $T$.
\end{theorem}

\begin{proof}
The reverse implication is the easy direction:
Suppose that $e$ is an edge such that the contracted tree $T^\prime = T/e$
has diameter less than the diameter of $T$.
The distance between vertices in a tree is given by the unique path that joins them,
and so contracting $e$ in $T$ reduces the distance of any pair by exactly $1$
in the case that $e$ belongs to the unique path joining the pair, and leaves the
distance unchanged otherwise. In particular, the difference between
the diameter of $T^\prime$ and the diameter of $T$ can be at most one,
and so $\diam(T^\prime) = \diam(T) - 1$.
The number of vertices also differs by exactly $1$,
and so the elementary minor $T^\prime$ of $T$ satisfies
\[
\uspcf{T^\prime}=|T^\prime|-\diam(T^\prime)-1 = |T|-\diam(T)-1 = \uspcf{T}.
\]

For the forward implication, there are three sorts of elementary minors.
Since $T$ is a connected tree, deletion of an isolated vertex can
only happen in the case $T=K_1$, leaving the empty graph as a minor.
Contraction of an edge $e$ that leaves the spectator number
unchanged must reduce the diameter by $1$ by the same calculation as above.
This leaves only the case of edge deletion.
Assume by way of contradiction, then, firstly that there does exist
a specific edge $e$ whose deletion produces
a disjoint union of trees $T_1$ and $T_2$ satisfying
\[
\uspcf{T_1 \cup T_2} = \uspcf{T_1} + \uspcf{T_2} = \uspcf{T},
\]
and secondly that no edge contraction reduces the diameter of $T$.
Using the diameter formula to substitute for 
$\uspcf{T_1}$, $\uspcf{T_2}$, and $\uspcf{T}$
produces an equation
\[
|T_1| - \diam(T_1) - 1 + |T_2| - \diam(T_2) - 1 = |T| - \diam(T) - 1
\]
whose simplification
\[
\diam(T) = \diam(T_1) + \diam(T_2) + 1
\]
implies the strict inequality
\[
\diam(T) > \diam(T_1)
\]
since $\diam(T_2) \ge 0$.

If no edge contraction reduces the diameter of $T$, then
in particular the contraction $T^\prime = T/e$ has the same diameter as $T$.
Let $p^\prime$ and $q^\prime$ be a diametric pair of vertices in $T^\prime$;
then $e$ cannot be part of the unique path that joins their preimages $p$ and $q$ in $T$.
It follows that $p$ and $q$ are in the same component $T_1$
or $T_2$ of $T\setminus e$. Without loss of generality, both are in $T_1$,
and thus that the diameter of $T_1$ is at least the diameter of
$T$:
\[
\diam(T) \le \diam(T_1).
\]
This contradicts the previous strict inequality and completes the proof.
\end{proof}

\begin{cor}
\label{cor:diam-paths}
A tree $T$ is minor-minimal for the spectator floor if and only if no edge lies
in the intersection of all diametric paths in $T$.
\end{cor}

\begin{proof}
The contraction of an edge $e$ reduces the diameter of $T$ if and only
$e$ lies in the intersection of all diametric paths in $T$,
and $T$ is minor-minimal for the spectator floor if and only if
no elementary minor of $T$ has the same spectator floor.
\end{proof}

As a result of \cref{cor:diam-paths}, we have the following, which tells us that star graphs are minor minimal for a given value of the spectator floor.

\begin{cor}
\label{cor:K1k+2}
For $k\geq1$, the graph $K_{1,k+2}$ is minor-minimal among graphs with spectator floor $k$.  
\end{cor}

\begin{proof}
By \cref{tree.uspcf}, $\uspcf{K_{1,k+2}}=k+3-2-1=k$. Since $k+2\geq3$, no edge of $K_{1,k+2}$ lies in all of the diametric paths of $K_{1.k+2}$. Therefore, by \cref{cor:diam-paths}, $K_{1,k+2}$ is minor-minimal among graphs with spectator floor $k$.
\end{proof}

In the final result of this section, we slightly improve upon the bound of \cref{diam.bound}.

\begin{theorem}
\label{lem:diameter}
For every graph $G$, we have $\uspcf{G}\geq|G|-\diam(G)-1$. 
Moreover, if $\usp{G}\leq\diam(G)$, then $\uspcf{G}\geq|G|-\diam(G)$.
\end{theorem}

\begin{proof}
By \cref{lem:no decontract}, there is a supergraph $G'$ of $G$ such that $\uspcf{G}=\uspc{G'}$ and $|G|=|G'|$. Since $G'$ is obtained from $G$ by adding edges but no vertices, we have $\diam(G')\leq\diam(G)$.

Recalling \cref{diam.bound}, we have $\uspcf{G}=\uspc{G'}\geq|G'|-\diam(G')-1\geq|G|-\diam(G)-1$. 

Now, suppose $\usp{G}\leq\diam(G)$. Consider a parade in $G'$, and let $u$ and $v$ be its endpoints. We will show that the length of this parade in $G'$ is at most $\diam(G)-1$.

If the distance between $u$ and $v$ in $G$ is at most $\diam(G)-1$, then the distance between $u$ and $v$ in $G'$ must be at most $\diam(G)-1$ also. On the other hand, consider the case that the distance in $G$ between $u$ and $v$ is $\diam(G)$. Since $\usp{G}\leq\diam(G)$, every parade in $G$ has at most $\diam(G)$ vertices, which implies that every parade in $G$ has length at most $\diam(G)-1$. Thus, the shortest paths joining $u$ and $v$ in $G$ are not unique. These paths are all still present in $G'$. Thus, since $u$ and $v$ are joined by a unique shortest path in $G'$, this path must have length at most $\diam(G)-1$.

Thus, we have $\usp{G'}\leq\diam(G)$, implying that $\uspcf{G}=\uspc{G'}\geq|G'|-\diam(G)=|G|-\diam(G)$.
\end{proof}

\section{Graphs of Spectator Floor 0, 1, and 2}
\label{sec:minimal}
If $A$ is the adjacency matrix of a graph $G$ and $k$ is a nonnegative integer, then the $(i, j)$-entry of $A^k$ counts the number of distinct walks of length exactly $k$ from vertex $i$ to vertex $j$ in $G$, including for example any paths of order $k + 1$.
%
%
%Given a multigraph $G$, the adjacency matrix $A$ is defined as the matrix whose $(i, j)$-entry counts the number of paths of length $1$ from vertex $i$ to vertex $j$.
%
%
%$G$ counts the number of edges between each pair of vertices, with diagonal entries counting the number of loops
%
%If $A$ is the adjacency matrix of $G$ and $k$ is a positive integer, then the $(i,j)$-entry of $A^k$ counts the number of distinct walks of exactly $k$ steps (involving $k + 1$ vertices, perhaps repeatedly) from vertex $i$ to vertex $j$ in graph $G$.
%In particular, if there exists a path of order $k + 1$ whose endpoints are vertex $i$ and vertex $j$, then 
This leads to the following efficient way of computing the parade number of a given graph.

\begin{obs}\label{USP-in-poly-time}
Let $G$ be a connected graph and let $k$ range from $0$ to $n$. Then \[\usp{G} =  1 + \max\{k: (A+2I)^k\text{ has a 1 in some entry}\}.\]
\end{obs}

We use $(A + 2I)^k$ rather than $A^k$ in order to include a contribution of at least $2A^j$ for all $j < k$, which ensures that an entry equal to $1$ represents not just a unique walk but a unique shortest walk, and therefore a unique shortest path. The calculation terminates once every entry is strictly greater than $1$. Recall that the binomial expansion of $(A+2I)^k$ includes all powers of $A$ from $A^0=I$ through $A^k$ so the $(i,j)$-entry of $(A+2I)^k$ is a weighted accumulation of the number of walks from $i$ to $j$ of length at most $k$. The need of $2I$ rather than $I$ can be seen by considering $K_1$; the parade number of $K_1$ is one but the adjacency matrix is $[0]$ and $([0] + I)^k = [1]$ for all $k$. Note that the choice of the multiplier 2 is arbitrary, any integer greater than one is sufficient.

    \textcolor{black}{The discussion above, along with Theorem \ref{lem:no decontract} and Corollary \ref{cor:components} led to algorithms which} were implemented in a SageMath \cite{sagemath} program to calculate the spectator floor of simple graphs and to determine which simple graphs are minor-minimal, with the code available at \cite{spec_floor_github_repo}. \textcolor{black}{These investigations helped lead to the characterizations in Corollary \ref{cor:min-multi-1}, Proposition \ref{prop:minimaml-simple}, and Theorem \ref{thm:min-multi-2}. However, these algorithms became very slow for graphs with large numbers of vertices. Therefore, an analytic approach was still needed.}

In \textcolor{black}{the remainder of} this section, we give the complete list of minor-minimal graphs of spectator floor 0, 1, and 2.
Along the way, we characterize those graphs with $\uspc{G} > 0$, $\uspc{G} > 1$, and $\uspc{G} > 2$, allowing quick recognition of such graphs.

We begin with graphs of spectator floor 0.

%\begin{lemma}
%\label{lem:loops}
%Let $L$ be the set of loops of a multigraph $G$. Then $\uspc{G\backslash L}=\uspc{G}$, and $\uspcf{G\backslash L}=\uspcf{G}$.
%\end{lemma}

%\begin{proof}
%It is clear that $\uspc{G\backslash L}=\uspc{G}$. Since $G\backslash L$ is a minor of $G$, we have $\uspcf{G\backslash L}\leq \uspcf{G}$. Now, let $H$ be a multigraph containing $G\backslash L$ as a minor such that $\uspcf{G\backslash L}=\uspc{H}$. By adding loops to $H$, we can obtain a multigraph $H^+$ containing $G$ as a minor. Therefore, we have $\uspcf{G\backslash L}=\uspc{H}=\uspc{H^+}\geq\uspcf{G}$.
%\end{proof}

\begin{prop}
\label{lem:path floor}
Let $G$ be a graph. Then $\uspcf{G}=0$ if and only if $G$ is a disjoint union of paths.
\end{prop}

\begin{proof}
%By \cref{lem:loops}, it suffices to consider loopless multigraphs.

By definition of $\uspc{G}$ and $\uspcf{G}$, it is clear that a path $P$ has $\uspc{P}=0$ and $\uspcf{P}=0$. 

First suppose $G$ is the disjoint union of paths $P_1,P_2,\ldots P_t$, and let the endpoints of $P_i$ be $x_i$ and $y_i$. Form a supergraph $G'$ of $G$ by adding edges joining $y_i$ and $x_{i+1}$, for $1\leq i\leq t-1$. Since $G'$ is a path, we have $\uspc{G'}=0$ and therefore $\uspcf{G}=0$.

Now, suppose $\uspcf{G}=0$. There must be a graph $G'$ such that $G$ is a minor of $G'$ and such that all vertices of $G'$ are contained in a unique shortest path $P$. Any edge added in parallel to an edge of $P$ causes $P$ to no longer be unique. If any other edge is added, it causes $P$ to not be a shortest path. Therefore, $G'$ must be a path. Since $G$ is a minor of a path, $G$ must be a disjoint union of paths.
\end{proof}

We now turn our attention to minor-minimal graphs with spectator floor $1$.  The first results is a corollary that follows from \cref{lem:path floor}.

\begin{cor}
\label{cor:min-multi-1}
The complete list of minor-minimal graphs with spectator floor $1$ is $C_2$ and $K_{1,3}$. The complete list of minor-minimal simple graphs with spectator floor $1$ is $K_3$ and $K_{1,3}$.
\end{cor}
\begin{proof}
For the first statement, if $G$ does not have spectator floor 0, then by \cref{lem:path floor} $G$ is not a disjoint union of paths, and so $G$ either has a vertex of degree at least $3$, or $G$ contains a cycle.  If $G$ has a vertex of degree at least $3$, then it contains $K_{1,3}$ as a subgraph; if $G$ contains a cycle, then $G$ con be contracted to $C_2$ (for multigraphs) or $K_3$ (for simple graphs).
\end{proof}

Next we begin our investigation of minor-minimal graphs with spectator floor 2.  The following lemma will be used later to prove that certain graphs have spectator floor 2.

\begin{lemma}
\label{lem:figures}
All of the graphs in \cref{fig:sharevertex,fig:shareedges,fig:one-cycle,fig:one-cycle2} have spectator floor $1$.
\end{lemma}

\begin{proof}
Consider the graphs in \cref{fig:sharevertex,fig:shareedges,fig:one-cycle,fig:one-cycle2}. (The vertex labels and captions will be used in the proof of \cref{thm:min-multi-2} below.) If we ignore the vertex labels, we note that all of these graphs are subgraphs of the graph in \cref{fig:sharevertex}. By  \cref{lem:path floor}, all of these graphs have spectator floor at least $1$. It is clear that the graph in \cref{fig:sharevertex} has spectator number $1$. Therefore, all of the graphs in \cref{fig:sharevertex,fig:shareedges,fig:one-cycle,fig:one-cycle2} have spectator floor at most $1$.
\end{proof}

%Moreover, the graph in \cref{fig:one-cycle2} is a subgraph of the graph in \cref{fig:sharevertex}. Therefore, all of the graphs in Figures \ref{fig:sharevertex}--\ref{fig:one-cycle2} have spectator floor $1$.
\begin{figure}[!htbp]
	\centering
	\begin{subfigure}[b]{.47 \textwidth}
	\centering

	\begin{tikzpicture}[x=.9cm, y=1cm]
		\vertex[fill,inner sep=1pt,minimum size=1pt] (a) at (0,0){};
		\vertex[fill,inner sep=1pt,minimum size=1pt] (b) at (1,0){};
		\vertex[fill,inner sep=1pt,minimum size=1pt] (c) at (2,0){};
		\vertex[fill,inner sep=1pt,minimum size=1pt] (u) at (3,0) 	[label=below:$u$] {};
		\vertex[fill,inner sep=1pt,minimum size=1pt] (w) at (4,0)	[label=below:$w$] {};
		\vertex[fill,inner sep=1pt,minimum size=1pt] (d) at (5,0){};
		\vertex[fill,inner sep=1pt,minimum size=1pt] (e) at (6,0){};
		\vertex[fill,inner sep=1pt,minimum size=1pt] (f) at (7,0){};
		\vertex[fill,inner sep=1pt,minimum size=1pt] (v) at (3.5,1) 	[label=above:$v$] {};
		\path
		(a) edge (b)
		(c) edge (u)
		(u) edge (w)
		(w) edge (d)
		(e) edge (f)
		(u) edge[bend right=20] (v)
		(u) edge[bend left=20] (v)
		(w) edge[bend right=20] (v)
		(w) edge[bend left=20] (v);
		\draw (b) to (1.3,0); \draw (1.7,0) to (c);
		\draw [dotted,semithick] (1.4,0) -- (1.65,0);
		\draw (d) to (5.3,0); \draw (5.7,0) to (e);
		\draw [dotted,semithick] (5.4,0) -- (5.65,0);
	\end{tikzpicture}
	\caption{$G$ if two cycles share exactly one vertex}
	\label{fig:sharevertex}
		\end{subfigure}
\begin{subfigure}[b]{.47\textwidth}
	\centering
	\begin{tikzpicture}[x=.9cm, y=1cm]
		\vertex[fill,inner sep=1pt,minimum size=1pt] (a) at (0,0){};
		\vertex[fill,inner sep=1pt,minimum size=1pt] (b) at (1,0){};
		\vertex[fill,inner sep=1pt,minimum size=1pt] (c) at (2,0){};
		\vertex[fill,inner sep=1pt,minimum size=1pt] (u) at (3,0) 	[label=below:$u$] {};
		\vertex[fill,inner sep=1pt,minimum size=1pt] (v) at (4,0)	[label=below:$v$] {};
		\vertex[fill,inner sep=1pt,minimum size=1pt] (d) at (5,0){};
		\vertex[fill,inner sep=1pt,minimum size=1pt] (e) at (6,0){};
		\vertex[fill,inner sep=1pt,minimum size=1pt] (f) at (7,0){};
		\vertex[fill,inner sep=1pt,minimum size=1pt] (w) at (3.5,1) 	[label=above:$w$] {};
		\path
		(a) edge (b)
		(c) edge (u)
		(u) edge (v)
		(w) edge (v)
		(v) edge (d)
		(e) edge (f)
		(u) edge[bend right=20] (w)
		(u) edge[bend left=20] (w);
		\draw (b) to (1.3,0); \draw (1.7,0) to (c);
		\draw [dotted,semithick] (1.4,0) -- (1.65,0);
		\draw (d) to (5.3,0); \draw (5.7,0) to (e);
		\draw [dotted,semithick] (5.4,0) -- (5.65,0);
	\end{tikzpicture}
	\caption{$G$ if it has at least two cycles}
	\label{fig:shareedges}
\end{subfigure}

\begin{subfigure}[b]{.47\textwidth}
	\centering
	\begin{tikzpicture}[x=.9cm, y=1cm]
		\vertex[fill,inner sep=1pt,minimum size=1pt] (a) at (0,0){};
		\vertex[fill,inner sep=1pt,minimum size=1pt] (b) at (1,0){};
		\vertex[fill,inner sep=1pt,minimum size=1pt] (c) at (2,0){};
		\vertex[fill,inner sep=1pt,minimum size=1pt] (u) at (3,0) 	[label=below:$u$] {};
		%\vertex[fill,inner sep=1pt,minimum size=1pt] (v) at (4,0)	[label=below:$v$] {};
		\vertex[fill,inner sep=1pt,minimum size=1pt] (d) at (4,0){};
		\vertex[fill,inner sep=1pt,minimum size=1pt] (e) at (5,0){};
		\vertex[fill,inner sep=1pt,minimum size=1pt] (f) at (6,0){};
		\vertex[fill,inner sep=1pt,minimum size=1pt] (v) at (3,1) 	[label=above:$v$] {};
		\path
		(a) edge (b)
		(c) edge (u)
		%(u) edge (v)
		%(w) edge (v)
		(u) edge (d)
		(e) edge (f)
		(u) edge[bend right=20] (v)
		(u) edge[bend left=20] (v);
		\draw (b) to (1.3,0); \draw (1.7,0) to (c);
		\draw [dotted,semithick] (1.4,0) -- (1.65,0);
		\draw (d) to (4.3,0); \draw (4.7,0) to (e);
		\draw [dotted,semithick] (4.4,0) -- (4.65,0);
	\end{tikzpicture}
	\caption{One possibility if $G$ has exactly one cycle}
	\label{fig:one-cycle}
\end{subfigure}
\begin{subfigure}[b]{.47\textwidth}
	\centering
	\begin{tikzpicture}[x=.9cm, y=1cm]
		\vertex[fill,inner sep=1pt,minimum size=1pt] (a) at (0,0){};
		\vertex[fill,inner sep=1pt,minimum size=1pt] (b) at (1,0){};
		\vertex[fill,inner sep=1pt,minimum size=1pt] (c) at (2,0){};
		\vertex[fill,inner sep=1pt,minimum size=1pt] (u) at (3,0) 	[label=below:$u$] {};
		\vertex[fill,inner sep=1pt,minimum size=1pt] (v) at (4,0)	[label=below:$v$] {};
		\vertex[fill,inner sep=1pt,minimum size=1pt] (d) at (5,0){};
		\vertex[fill,inner sep=1pt,minimum size=1pt] (e) at (6,0){};
		\vertex[fill,inner sep=1pt,minimum size=1pt] (f) at (7,0){};
		\path
		(a) edge (b)
		(c) edge (u)
		(v) edge (d)
		(e) edge (f)
		(u) edge[bend right=20] (v)
		(u) edge[bend left=20] (v);
		\draw (b) to (1.3,0); \draw (1.7,0) to (c);
		\draw [dotted,semithick] (1.4,0) -- (1.65,0);
		\draw (d) to (5.3,0); \draw (5.7,0) to (e);
		\draw [dotted,semithick] (5.4,0) -- (5.65,0);
	\end{tikzpicture}
	\caption{Another possibility if $G$ has exactly one cycle}
	\label{fig:one-cycle2}
\end{subfigure}
\caption{Some graphs with spectator floor 1}
\label{fig:specfloor1}
\end{figure}
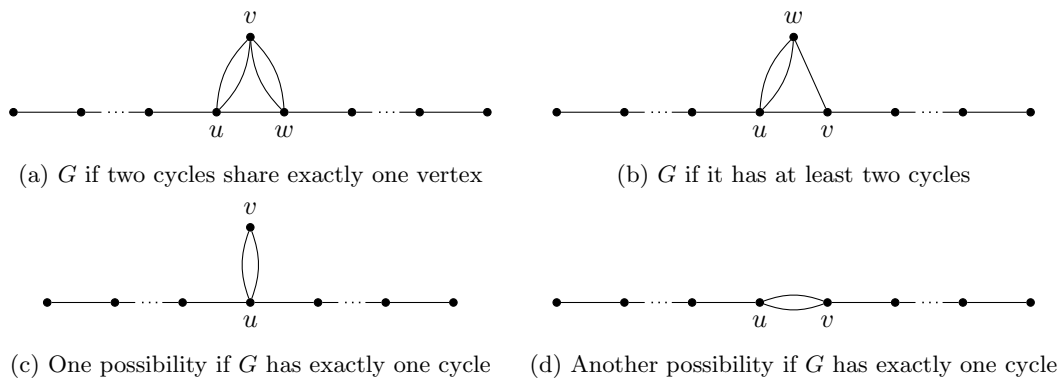

Our first pair of results for spectator floor 2 are concerning simple graphs.  The lemma below gives a list of minor-minimal simple graphs with spectator floor 2, and in the following \cref{prop:minimaml-simple} we will also show that this is the complete list of such graphs.

\begin{lemma}
\label{lem:simple-minimal}
The following graphs are minor-minimal among simple graphs with spectator floor $2$: $K_3\sqcup K_3$, $K_3\sqcup K_{1,3}$, $K_{1,3}\sqcup K_{1,3}$, $C_4$, $K_{1,4}$, the $3$-sun (see \cref{fig:3sun}), and the long $Y$ (see \cref{fig:longY}). %\kevin{I Googled ``3-sun graph" and found something different. We might want to call it something else.}
\end{lemma}

\begin{figure}[!htbp]
\begin{subfigure}[b]{.47\textwidth}
\[\begin{tikzpicture}[x=1cm, y=1cm]
\vertex[fill,inner sep=1pt,minimum size=1pt] (a) at (0,0) 		[label=above:$$] {};
\vertex[fill,inner sep=1pt,minimum size=1pt] (b) at (1,0) 		[label=left:$$] {};
\vertex[fill,inner sep=1pt,minimum size=1pt] (c) at (2,0) 	[label=below:$$] {};
\vertex[fill,inner sep=1pt,minimum size=1pt] (d) at (3,0) 	[label=below:$$] {};
\vertex[fill,inner sep=1pt,minimum size=1pt] (e) at (1.5,1)	[label=above:$$] {};
\vertex[fill,inner sep=1pt,minimum size=1pt] (f) at (1.5,2) 	[label=below:$$] {};
\path
(a) edge (b)
(b) edge (c)
(c) edge (d)
(b) edge (e)
(c) edge (e)
(e) edge (f);
\end{tikzpicture}\]
\caption{The $3$-sun}
\label{fig:3sun}
\end{subfigure}
\begin{subfigure}[b]{.47\textwidth}
\[\begin{tikzpicture}[x=1cm, y=1cm]
\vertex[fill,inner sep=1pt,minimum size=1pt] (a) at (0,0) 		[label=above:$$] {};
\vertex[fill,inner sep=1pt,minimum size=1pt] (b) at (0,-1) 		[label=left:$$] {};
\vertex[fill,inner sep=1pt,minimum size=1pt] (c) at (0,-2) 	[label=below:$$] {};
\vertex[fill,inner sep=1pt,minimum size=1pt] (d) at (-.7,.7) 	[label=below:$$] {};
\vertex[fill,inner sep=1pt,minimum size=1pt] (e) at (-1.4,1.4)	[label=above:$$] {};
\vertex[fill,inner sep=1pt,minimum size=1pt] (f) at (.7,.7) 	[label=below:$$] {};
\vertex[fill,inner sep=1pt,minimum size=1pt] (g) at (1.4,1.4)	[label=above:] {};
\path
(a) edge (b)
(b) edge (c)
(a) edge (d)
(d) edge (e)
(a) edge (f)
(f) edge (g);
\end{tikzpicture}\]
\caption{The long $Y$}
\label{fig:longY}
\end{subfigure}
\caption{Two graphs from \cref{lem:simple-minimal}}

\end{figure}
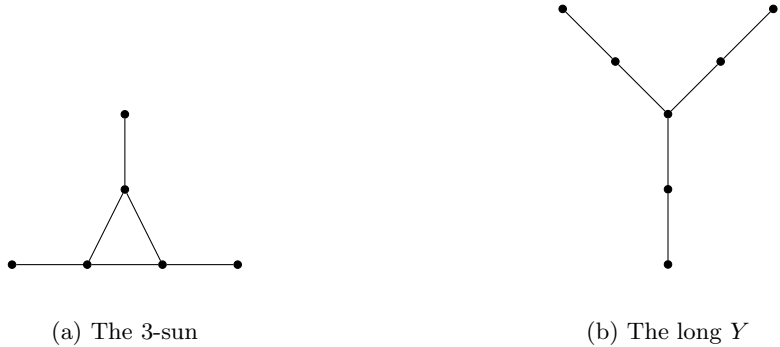

\begin{proof}
It follows from \cref{prop:components} that a disconnected graph is minor-minimal among simple graphs with spectator floor $2$ if and only if it has exactly two components, each of which have spectator floor $1$. Therefore, by \cref{cor:min-multi-1}, $K_3\sqcup K_3$, $K_3\sqcup K_{1,3}$, $K_{1,3}\sqcup K_{1,3}$ are all minor-minimal among simple graphs with spectator floor $2$.

The fact that $K_{1,4}$ is minor-minimal among simple graphs with spectator floor $2$ follows directly from \cref{cor:K1k+2}. It follows from \cref{tree.uspcf,cor:diam-paths} that the long $Y$ is also minor-minimal among simple graphs with spectator floor $2$.

Note that $\usp{C_4}=\diam(C_4)=2$. By the second statement of \cref{lem:diameter}, we have $\uspcf{C_4}\geq4-2=2$. Since $\uspcf{C_4}\geq\usp{C_4}=2$, we have $\uspcf{C_4}=2$. To see that $C_4$ is minor-minimal, note that deletion of any edge, results in a path, which has spectator floor $0$ and that contraction of any edge results in a triangle, which has spectator floor $1$.

Let $G$ be the $3$-sun, and note that $|G|=6$ and $\diam(G)=3$. By the first statement in \cref{lem:diameter}, we have $\uspcf{G}\geq6-3-1=2$. Since $\uspcf{G}\leq\uspc{G}=2$, we have $\uspcf{G}=2$.

To see that the $3$-sun is minor-minimal, first recall from \cref{lem:isolated} that isolated vertices have no effect on the spectator floor of a graph. If we delete any edge from the $3$-sun and then disregard any isolated vertices that may result, we obtain a subgraph of a graph of the form given in \cref{fig:sharevertex}. Thus, deletion of any edge of the $3$-sun results in a graph with spectator floor $1$.

We now consider the effect of contracting an edge from the $3$-sun. Since we want to show that the $3$-sun is minor-minimal among simple graphs, we should immediately simplify once if any parallel edges result from the contraction. One can easily check that, if any edge is contracted and the resulting graph is simplified, then the result is a subgraph of a graph of the form given in \cref{fig:sharevertex}. Thus, the resulting graph has spectator floor $1$.
\end{proof}

We are now prepared to show that the list of graphs in \cref{lem:simple-minimal} is in fact the complete list of minor-minimal simple graphs with spectator floor 2.

\begin{prop}
\label{prop:minimaml-simple}
The complete list of minor-minimal (simple) graphs with spectator floor $2$ is $K_3\sqcup K_3$, $K_3\sqcup K_{1,3}$, $K_{1,3}\sqcup K_{1,3}$, $C_4$, $K_{1,4}$, the long $Y$, and the $3$-sun.
\end{prop}

\begin{proof}
Suppose for a contradiction that $G$ is a minor-minimal simple graph with spectator floor $2$. If $G$ is not connected, then by \cref{prop:components}, $G$ is the disjoint union of graphs $G_1$ and $G_2$, each with spectator floor $1$. The minor-minimal graphs with spectator floor $1$ are $K_3$ and $K_{1,3}$. Therefore, $G$ is $K_3\sqcup K_3$, $K_3\sqcup K_{1,3}$, or $K_{1,3}\sqcup K_{1,3}$. Thus, we may assume that $G$ is connected.

Since $G$ is minor-minimal, it does not contain $C_4$ or $K_{1,4}$ as a minor. Therefore, the maximum degree of $G$ is $3$, and $G$ has no cycle of length greater than $3$.

We now show that $G$ has at most one triangle. If $G$ has two disjoint triangles, then $G$ has $K_3\sqcup K_3$ as a minor. If two triangles of $G$ share exactly one vertex, then that vertex has degree $4$. If two triangles share an edge, then $G$ contains  $C_4$. Therefore, $G$ has at most one triangle.

If $G$ has two vertices of degree $3$ that are not contained in a triangle, then by contracting a path joining the vertices, we obtain $K_{1,4}$ as a minor. Thus, either

\begin{itemize}
\item[(i)] $G$ has exactly one triangle, and all vertices not in the triangle have degree $1$ or $2$, or
\item[(ii)] $G$ is a tree with at most one vertex of degree $3$, and all other vertices of $G$ have degree $1$ or $2$.
\end{itemize}

We first consider case (i). If every vertex of the triangle has degree $3$, then $G$ has the $3$-sun as a minor. Otherwise, $G$ consists of a path with one additional vertex forming a triangle with two vertices in the path. Then $\uspc{G}=1$, and we have a contradiction.

Now we consider case (ii). If $G$ has no vertex of degree $3$, then $G$ is a path, and $\uspc{G}=0$, a contradiction. Thus, we may assume that $G$ has exactly one vertex $v$ of degree $3$. If all three vertices adjacent to $v$ have degree $2$, then $G$ has the long $Y$ as a minor. Otherwise, $G$ consists of a path with one additional vertex adjacent to one vertex on the path. Then $\uspc{G}=1$, and we have a contradiction.
\end{proof}

Our next pair of results for spectator floor 2 are concerning the more general case of multigraphs.  We first need to define the graphs $H_1$ through $H_5$.

\begin{definition}
\label{def:minor-minimal}
 Let $H_1$ be obtained from $K_3$ by doubling every edge. Let $H_2$ be obtained from $K_{1,3}$ by doubling two of the edges. Let $H_3$ be the $1$-sum of $K_3$ and $C_2$, and let $H_4$ be obtained by contracting one edge of the triangle in the $3$-sun. Let $H_5$ be the graph shown in \cref{fig:rocket}.
\begin{figure}[!htbp]
\[\begin{tikzpicture}[x=1cm, y=1cm]
\vertex[fill,inner sep=1pt,minimum size=1pt] (a) at (0,0.5){};
\vertex[fill,inner sep=1pt,minimum size=1pt] (b) at (0,1.5){};
\vertex[fill,inner sep=1pt,minimum size=1pt] (c) at (1,1){};
\vertex[fill,inner sep=1pt,minimum size=1pt] (d) at (2,1){};
\vertex[fill,inner sep=1pt,minimum size=1pt] (e) at (3,1){};
\path
(a) edge (c)
(b) edge (c)
(d) edge (c)
(e) edge[bend right=20] (d)
(e) edge[bend left=20] (d);
\end{tikzpicture}\]
\caption{$H_5$}
\label{fig:rocket}
\end{figure}
\end{definition}

The lemma below gives a list of minor-minimal (multi)graphs with spectator floor 2, and in the following \cref{thm:min-multi-2} we will also show that this is the complete list of such graphs.

\begin{lemma}
\label{lem:multi-minimal}
The following graphs are all minor-minimal among graphs with spectator floor $2$: $C_2\sqcup C_2$, $C_2\sqcup K_{1,3}$, $K_{1,3}\sqcup K_{1,3}$, $C_4$, $H_1$, $H_2$, $H_3$, $H_4$, $H_5$, $K_{1,4}$, and the long $Y$.
\end{lemma}

\begin{proof}
It follows from \cref{prop:components} that a disconnected graph is minor-minimal among simple graphs with spectator floor $2$ if and only if it has exactly two components, each of which have spectator floor $1$. Therefore, by \cref{cor:min-multi-1}, $C_2\sqcup C_2$, $C_2\sqcup K_{1,3}$, and $K_{1,3}\sqcup K_{1,3}$ are all minor-minimal among graphs with spectator floor $2$.

Every single-edge deletion and every single-edge contraction of $C_4$, $K_{1,4}$, and the long $Y$ is a simple graph. Therefore, \cref{lem:simple-minimal} implies that these graphs are minor-minimal among graphs with spectator floor $2$.

One can check that $\diam(H_1)=1$, that $\diam(H_2)=\diam(H_3)=2$, and that $\diam(H_4)=\diam(H_5)=3$. One can also check that $\usp{H_1}=1$, that $\usp{H_2}=\usp{H_3}=2$, and that $\usp{H_4}=\usp{H_5}=3$. It follows that, if $G\in\{H_1,H_2,H_3,H_4,H_5\}$, then $\uspc{G}=2$. Thus, $\uspcf{G}\leq2$. Moreover, since $\usp{G}=\diam(G)$, \cref{lem:diameter} implies that $\uspcf{G}\geq|G|-\diam(G)=2$. Therefore, $\uspcf{G}=2$.

To show that $H_1$, $H_2$, $H_3$, $H_4$, and $H_5$ are minor-minimal, we must show that every single-edge deletion and every single-edge contraction from each of these graphs is a graph with spectator number less than $2$. Since every edge of $H_1$ is in parallel with another edge, no edge can be contracted. If an edge is deleted from $H_1$, then the resulting graph is a subgraph of a graph of the form given in \cref{fig:sharevertex}, which has spectator number $1$. If $G\in\{H_2,H_3,H_4,H_5\}$, then one can check that every single-edge deletion and every single-edge contraction of $G$ is a subgraph of a graph of the form given in \cref{fig:sharevertex}, which has spectator number $1$.
\end{proof}

Next, we have a couple of technical lemmas that will support the proof of \cref{thm:min-multi-2}.

\begin{lemma}
\label{lem:max2edges}
Let $G$ be a graph such that two or more edges join vertices $v$ and $w$. If $e$ is one of these edges, then $\usp{G\backslash e}\geq \usp{G}$ and $\uspc{G\backslash e}\leq \uspc{G}$. Moreover, if three or more edges join vertices $v$ and $w$, then $\usp{G\backslash e}=\usp{G}$ and $\uspc{G\backslash e}=\uspc{G}$.
\end{lemma}

\begin{proof}
No unique shortest path in $G$ contains $e$ since there is at least one other edge in parallel with $e$. Therefore, every unique shortest path in $G$ is also a unique shortest path in $G\backslash e$.

The only case where a unique shortest path in $G\backslash e$ is not a unique shortest path in $G$ is if $G\backslash e$ has exactly one edge joining $v$ and $w$ and this path contains that edge.

Therefore, the set of unique shortest paths of $G$ is a subset of the set of unique shortest paths of $G\backslash e$, implying that $\usp{G\backslash e}\geq \usp{G}$ and $\uspc{G\backslash e}\leq \uspc{G}$.

Moreover, if three or more edges join $v$ and $w$ in $G$, then two or more edges join $v$ and $w$ in $G\backslash e$. Therefore, the set of unique shortest paths of $G$ is equal to the set of unique shortest paths of $G\backslash e$, implying that $\usp{G\backslash e}=\usp{G}$ and $\uspc{G\backslash e}=\uspc{G}$.
%there is if this path contains the edge joining $v$ and $w$
%
%However, a unique shortest path in $G$ that contains an edge joining $v$ and $w$ is not a unique shortest path
%
%is also a parade in 
%
%Every shortest path between a pair of vertices in $G\backslash e$ is also a shortest path in $G$. The only shortest paths between pairs of vertices in $G$ that are not shortest paths in $G\backslash e$ are paths that contain $e$. However, these paths are not unique in $G$, and the paths obtained from these paths by replacing $e$ with a different edge joining $v$ and $w$ are also not unique in $G$.
%
%or $G\backslash e$ (since there are at least two other edges joining $v$ and $w$). Therefore, both graphs have exactly the same parades. This implies that they have the same parade number and spectator number.
\end{proof}

\begin{lemma}
\label{lem:max2edgesfloor}
Let $G$ be a graph such that three or more edges join vertices $v$ and $w$. If $e$ is one of these edges, then $\uspcf{G\backslash e}=\uspcf{G}$.
\end{lemma}

\begin{proof}
Since $G\backslash e$ is a minor of $G$, we have $\uspcf{G\backslash e}\leq\uspcf{G}$.

Now, let $H$ be a graph containing $G\backslash e$ as a minor such that $\uspcf{G\backslash e}=\uspc{H}$. By \cref{lem:no decontract}, we may assume that $G\backslash e$ is obtained from $H$ by deleting edges. Therefore, there are at least two edges joining $v$ and $w$ in $H$. Add an additional edge joining $v$ and $w$ in $H$ to form the graph $H^+$, which contains $G$ as a minor. By \cref{lem:max2edges}, we have $\uspc{H}=\uspc{H^+}$. We then have $\uspcf{G\backslash e}=\uspc{H}=\uspc{H^+}\geq\uspcf{G}$
\end{proof}

In our final result for this section, we are able to prove that the list of graphs in \cref{lem:multi-minimal} is in fact the complete list of minor-minimal (multi)graphs with spectator floor 2.

\begin{theorem}
\label{thm:min-multi-2}
The complete list of minor-minimal graphs with spectator floor $2$ is $C_2\sqcup C_2$, $C_2\sqcup K_{1,3}$, $K_{1,3}\sqcup K_{1,3}$, $C_4$, $H_1$, $H_2$, $H_3$, $H_4$, $H_5$, $K_{1,4}$, and the long $Y$.
\end{theorem}

\begin{proof}
Let $G$ be a minor-minimal graph with spectator floor $2$, and suppose for a contradiction that $G$ is not one of the graphs given in the statement of the result. %By \cref{lem:loops}, $G$ has no loops.
By \cref{lem:max2edgesfloor}, there are at most two edges joining each pair of vertices of $G$.

\begin{claim}
\label{connected}
$G$ is connected.
\end{claim}

\begin{subproof}
If $G$ is not connected, then by \cref{prop:components}, $G$ is the disjoint union of graphs $G_1$ and $G_2$, each with spectator floor $1$. By \cref{cor:min-multi-1}, the minor-minimal graphs with spectator floor $1$ are $C_2$ and $K_{1,3}$. Therefore, $G$ is $C_2\sqcup C_2$, $C_2\sqcup K_{1,3}$, or $K_{1,3}\sqcup K_{1,3}$, each of which are graphs given in the statement of the result.
\end{subproof}

Since $G$ does not contain $C_4$ as a minor, we have the following.
\begin{claim}
\label{cycle}
$G$ has no cycle of length at least $4$.
\end{claim}

Since $G$ does not contain $K_{1,4}$ as a minor, we have the following.
\begin{claim}
\label{neighborhood}
$G$ has no vertex with a neighborhood of cardinality at least $4$.
\end{claim}

\begin{claim}
\label{vertex}
$G$ has no pair of disjoint cycles and no pair of cycles that share exactly one vertex.
\end{claim}

\begin{subproof}
Since $C_2\sqcup C_2$ is not a minor of $G$, there is no pair of disjoint cycles in $G$. Since $H_3$ is not a minor of $G$, if two cycles share exactly one vertex, both cycles must have length $2$.

Now, suppose for a contradiction that $G$ contains two copies of $C_2$ as subgraphs and that these copies share exactly one vertex $v$. Let $u$ and $w$ be the other vertices in these copies of $C_2$. Because $H_2$ is not a minor of $G$, we have $N_G(v)=\{u,w\}$. Because $H_4$ is not a minor of $G$, we have $|N_G(u)-\{v,w\}|\leq1$ and $|N_G(w)-\{u,v\}|\leq1$. Moreover, since $C_4$ and $H_1$ are not minors of $G$, there is no path from $u$ to $w$ in $G-v$ except possibly at most one edge joining $u$ and $w$. Finally, because $H_5$ is not a minor of $G$, no vertex in $V(G)-\{u,v,w\}$ has a neighborhood of cardinality greater than $2$. Therefore, $G$ is a subgraph of the graph in \cref{fig:sharevertex}.
This graph has spectator number $1$, Thus $\uspcf{G}=1$, a contradiction.
\end{subproof}

\begin{claim}
\label{triangles}
No pair of triangles of $G$ share exactly one edge.
\end{claim}

\begin{subproof}
Otherwise, the union of these triangles contains a cycle of length $4$, violating \cref{cycle}. 
\end{subproof}

\begin{claim}
\label{onecycle}
$G$ has at most one cycle.
\end{claim}

\begin{subproof}
Suppose for a contradiction that $G$ has more than one cycle. We know that %$G$ has no loops and that
each pair of vertices of $G$ is joined by at most two edges. Therefore, \cref{cycle,vertex,triangles} imply that $G$ contains a triangle with vertices $u$, $v$, and $w$ with a second edge joining $u$ and $w$.

By \cref{vertex}, there is no path in $G-w$ from $u$ to $v$ other than the edge $uv$. Similarly, there is no path in $G-u$ from $w$ to $v$ other than the edge $wv$. We can also see that there is no path from $u$ to $w$ in $G-v$ other than the two edges joining $u$ and $w$. This is because $G$ has no cycle of length at least $4$ and at most two edges joining $u$ and $w$.

Therefore, there are subgraphs $G_u$, $G_v$, and $G_w$ of $G$ such that, for each vertex $x$ in $G_i$, the path from $x$ to $\{u,v,w\}$ has endpoints $x$ and $i$. (For each $i\in\{u,v,w\}$, we have $i\in V(G_i)$.) This path must be unique; otherwise, \cref{vertex} is violated. Thus, each of $G_u$, $G_v$, and $G_w$ is a tree. Moreover, if we denote by $F$ the set of four edges both of whose endpoints are in $\{u,v,w\}$, then $G$ is the graph whose vertex set is $V(G_u)\sqcup V(G_v)\sqcup V(G_w)$ and whose edge set is $E(G_u)\sqcup E(G_v)\sqcup E(G_w)\sqcup F$.

For $i\in\{u,v,w\}$, vertex $i$ has at most one neighbor in $G_i$. Otherwise, $G$ has $K_{1,4}$ as a minor. Similarly, since $G$ does not have $K_{1,4}$ as a minor, no vertex in $V(G_i)-\{i\}$ has more than two neighbors.

If $u$ and $w$ have neighbors in $G_u$ and $G_w$, respectively, then $G$ contains $H_4$ as a minor. Therefore, either $N_G(u)=\{v,w\}$ or $N_G(w)=\{u,v\}$. Without loss of generality, let $N_G(w)=\{u,v\}$. Then $G$ is a subgraph of a graph of the form given in \cref{fig:shareedges}. 
This graph has spectator number $1$, Thus $\uspcf{G}=1$, a contradiction.
\end{subproof}

Therefore, one of the following holds:
\begin{itemize}
\item[(i)] $G$ is a tree, 
\item[(ii)] $G$ has exactly one cycle, which is a triangle, or
\item[(iii)] $G$ has exactly one cycle, which is a $C_2$.
\end{itemize}

In cases (i) and (ii), $G$ is a simple graph. Since $G$ is not any of the graphs listed above, and since $G$ is connected, \cref{prop:minimaml-simple} implies that $G$ is the $3$-sun. However, by contracting an edge of the triangle in the $3$-sun, we obtain $H_4$. Therefore, $G$ has exactly one cycle, which is a $C_2$.

Let $u$ and $v$ be the vertices of $C_2$. Because $H_5$ is not a minor of $G$, every vertex in $V(G)-\{u,v\}$ has degree $1$ or $2$. Since $K_{1,4}$ is not a minor of $G$, we have $|N_G(u)-\{v\}|\leq2$ and  $|N_G(v)-\{u\}|\leq2$. Moreover, either $|N_G(u)-\{v\}|<2$ or $|N_G(v)-\{u\}|<2$. Without loss of generality, let $|N_G(v)-\{u\}|<2$.

If $|N_G(v)-\{u\}|=0$, then $G$ is a subgraph of a graph of the form shown in \cref{fig:one-cycle}.
Thus, we have have $\uspc{G}=\uspcf{G}=1$, a contradiction. 

If $|N_G(v)-\{u\}|=1$, then since $H_4$ is not a subgraph of $G$, we must also have $|N_G(u)-\{v\}|=1$. Therefore, $G$ is a subgraph of a graph of the form given in \cref{fig:one-cycle2},
which is isomorphic to a subgraph of a graph of the form shown in \cref{fig:shareedges}. The graph in \cref{fig:shareedges} has spectator number $1$. Therefore, $\uspcf{G}=1$, a contradiction.

Therefore, by contradiction, we conclude that the result holds.
\end{proof}

\section{Minor Maximal Graphs}
\label{minor max graphs}
To this point in the paper, we have discussed minor minimal graphs at some length; in this sense, we have only looked downwards.  We shall now look upwards and consider the possibilities of minor maximal graphs of a given spectator floor.

We begin with the observation that, without additional restrictions, there are no minor maximal graphs with a given spectator floor.  This is so because, given any graph $G$, if we add a new isolated vertex to $G$ in order to obtain $G'$, then $G$ is a minor of $G'$, and by \cref{prop:components}, $G'$ has the same spectator floor as $G$.  Hence, we can construct an infinite chain of graphs which are above $G$ and have the same spectator floor.

Therefore, in order to obtain any meaningful information, we must search for minor maximal graphs amongst subsets of graphs which are restricted in some way.  The most natural restriction to make is on the number of vertices and on the number of parallel edges allowed, for which we can completely characterize the minor maximal graphs.

In order to present the result, we first must present new definitions.

\begin{definition}\label{def:crowded_parade}
If $p\geq2$, then a \emph{crowded $p$-parade} is a graph $G$ with a parade of $p$ vertices that has the following properties:
\begin{itemize}
    \item Every vertex outside the parade is connected to exactly two vertices in the parade, and those two parade vertices are adjacent to each other.
    \item Any two vertices outside the parade that are adjacent to a common parade vertex are also adjacent to each other.
\end{itemize}
A \emph{crowded $1$-parade} is a graph whose simplification is a complete graph.
\end{definition}

\textcolor{black}{Since a parade is a unique shortest path between two vertices, it follows that two non-parade vertices of a crowded parade can only be adjacent if they share a neighbor on the parade.}

\begin{definition}\label{def:sat_crowded_parade}
If $p\geq2$, an \emph{$m$-saturated crowded $p$-parade} is a crowded $p$-parade where every edge which is not in the parade has $m$ parallel copies, including itself. An \emph{$m$-saturated crowded $1$-parade} is a graph such that every pair of vertices is joined by exactly $m$ edges.
\end{definition}

\begin{figure}[!htbp]
\[\begin{tikzpicture}[x=1cm, y=1cm]
\vertex[fill] (a) at (0,0){};
\vertex[fill] (b) at (2,0){};
\vertex[fill] (c) at (4,0){};
\vertex[fill] (d) at (6,0){};
\vertex[fill] (e) at (8,0){};
\vertex[fill] (f) at (10,0){};
\vertex[fill] (g) at (12,0){};

\vertex[fill] (bc) at (3,1.5){};
\vertex[fill] (cd1) at (5,2){};
\vertex[fill] (cd2) at (5,1){};
\vertex[fill] (de) at (7,1.5){};
\vertex[fill] (fg1) at (11,2){};
\vertex[fill] (fg2) at (10.4,1.5){};
\vertex[fill] (fg3) at (11.6,1.5){};
\draw (a) to (g);
\draw[double distance=2] (b) to (bc) to (c);
\draw[double distance=2] (cd1) to (c) to (cd2) to (d) to (cd1);
\draw[double distance=2] (cd1) to (cd2);
\draw[double distance=2] (d) to (de) to (e);
\draw[double distance=2] (bc) to (cd1) to (de) to (cd2) to (bc);
\draw[double distance=2] (f) to (fg1) to (g);
\draw[double distance=2] (f) to (fg2) to (g);
\draw[double distance=2] (f) to (fg3) to (g);
\draw[double distance=2] (fg1) to (fg2) to (fg3) to (fg1);
\end{tikzpicture}\]
\caption{An example of a 2-saturated crowded 7-parade.}
\label{fig:crowded_parade}
\end{figure}
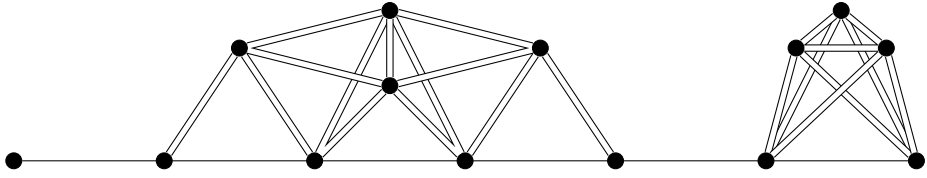

With these definitions in hand, we will show that a graph is minor maximal amongst graphs of a given spectator floor value \textcolor{black}{and given number of vertices} if and only if the graph is a crowded parade.

Before presenting this characterization of minor maximal graphs, however, we first have two technical lemmas that we will use repeatedly in the proof of the characterization.  In the first lemma, we use the notation $d_S(x,y)$ to refer to the distance between \textcolor{black}{vertices $x$ and $y$ in graph $S$}.

\begin{lemma}\label{lem:maximals}
Let $G$ be a graph that is minor maximal amongst graphs with $n$ vertices, at most $m$ parallel edges between any given vertices, and spectator floor $k$.  Let $a,b$ be the endpoints of a parade in $G$.  Let $H$ be $G$ with one edge added between some vertices $v$ and $w$ where no edge was present in $G$.  Then at least one of the following is true:
\begin{align*}
    d_H(a,b)&=d_G(a,v)+1+d_G(w,b), \ \ \mbox{or} \\
    d_H(a,b)&=d_G(a,w)+1+d_G(v,b).
\end{align*}
\end{lemma}
\begin{proof}
By \cref{lem:no decontract}, there is a graph $G'$ with the same number of vertices as $G$ such that $G$ is a subgraph of $G'$ and such that $\uspcf{G}=\uspc{G'}$.  Further, we may also assume that $G'$ has at most $m$ parallel edges between any given vertices. (It follows from \cref{lem:max2edges} that adding an edge in parallel with an edge in $G$ can only increase the spectator number.) Since we assumed that $G$ is maximal amongst such graphs, we conclude that $G'=G$.  Hence $\uspc{G}=k$.

Now consider the graph $H$, which is the same as $G$ but with an edge added between $v$ and $w$, which are some vertices of $G$ that are not adjacent in $G$.  Since $G$ is a proper subgraph of $H$, and $H$ is a graph with $n$ vertices and at most $m$ parallel edges between any given vertices, and $G$ is maximal amongst graphs with $n$ vertices, at most $m$ parallel edges between any given vertices, and spectator floor $k$, we conclude that $H$ does not have spectator floor $k$.  Furthermore, since $G$ is a subgraph of $H$, $\uspcf{H}>\uspcf{G}$.  Thus, $\uspc{H}>\uspc{G}$.

Since $\uspc{G}=k$, $G$ has a parade $P$ of $n-k$ vertices.  Let us call the endpoints of $P$ by the names $a$ and $b$.  Since $H$ still contains $P$ but $\uspc{H}\neq k$, we conclude that $P$ is no longer a longest unique shortest path in $H$.  There cannot be a longer unique shortest path than $P$ in $H$, however, since that would result in $\uspc{H}<\uspc{G}$, so it must be that $P$ is not a unique shortest path in $H$.  There are two possibilities: either $P$ is still a shortest path in $H$ but no longer unique, or $P$ is no longer a shortest path in $H$.  Either way, there must be a new shortest path in $H$ between $a$ and $b$ that is of the same or shorter length than $P$; let us call this new path $Q$.

$Q$ must contain the edge $\{v,w\}$, since that edge is the only difference between $G$ and $H$.  There are two possibilities.  If $Q$ connects $a,v,w,b$ in that order, then we have
\[d_H(a,b)=d_G(a,v)+1+d_G(w,b).\]
Otherwise, if $Q$ connects $a,w,v,b$ in that order, then we have
\[d_H(a,b)=d_G(a,w)+1+d_G(v,b).\]
\end{proof}

We now present another technical lemma to be used in proving the characterization of minor maximal graphs.  This lemma tells us that minor maximal graphs must be connected.

\begin{lemma}\label{lem:maxls_are_connected}
Let $G$ be a graph that is minor maximal amongst graphs with $n$ vertices, at most $m$ parallel edges between any given vertices, and spectator floor $k$.  Then $G$ is connected.
\end{lemma}
\begin{proof}
Suppose to the contrary that $G$ is not connected, but rather consists of connected components $G_1,G_2,...,G_r$.  We know from the proof of \cref{lem:maximals} that $\uspc{G}=\uspcf{G}=k$, and from \cref{prop:components} we know that
\[\uspcf{G}=\uspcf{G_1}+\uspcf{G_2}+\cdots+\uspcf{G_r}.\]

For each $i=1,2,...,r$, let $H_i$ be a supergraph of $G_i$ that realizes $\uspcf{G_i}$ without adding any additional vertices; that is, $\uspc{H_i}=\uspcf{G_i}$.  We know such $H_i$ exist from \cref{{lem:no decontract}}.  Then let $H$ be a supergraph of $H_1 \cup H_2 \cup \cdots \cup H_r$ in which we add $r-1$ edges to the graph in order to take one parade from each of the $H_i$ and connect them into one long parade.

Then we have that $H$ is a supergraph of $G$ and $H$ still has $n$ vertices.  Furthermore, due to the way we constructed $H$, we have 
\begin{align*}
    \uspc{H} &=\uspc{H_1}+\uspc{H_2}+\cdots+\uspc{H_r} \\
            & =\uspcf{G_1}+\uspcf{G_2}+\cdots+\uspcf{G_r} \\
            & =\uspcf{G}=\uspc{G}.
\end{align*}

Since $H$ is a supergraph of $G$, it follows that $\uspcf{H}\geq \uspcf{G}=\uspc{G}$.  Furthermore, since $\uspc{H}=\uspc{G}$, it follows that $\uspcf{H}\leq \uspc{G}$.  Hence $\uspcf{H}=\uspc{G}=k$.

Thus, since $H$ has $n$ vertices and spectator floor $k$, and is a supergraph of $G$, and since we assumed $G$ was maximal amongst such graphs, it must be that $G=H$.  Since we assumed that $G$ was disconnected and $H$ is connected by construction, this is a contradiction.
\end{proof}

Before presenting the main results of this section, we prove a lemma that takes care of a special case.

\begin{lemma}
\label{lem:1-parade} Let $m\geq2$. Then $G$ is minor maximal amongst graphs with $n$ vertices, at most $m$ parallel edges between any given pair of vertices, and spectator floor $n-1$, if and only if $G$ is an $m$-saturated crowded $1$-parade.
\end{lemma}

\begin{proof}
Note that there is exactly one $m$-saturated crowded $1$-parade $H$ with $n$ vertices. Since $m\neq1$, no edge in $H$ is a unique shortest path. This implies that the unique shortest paths of $H$ each contain only one vertex. Therefore, $\uspc{H}=1$, implying that $\uspcf{H}=1$. Every graph containing $H$ as a minor either has more than $n$ vertices or a pair of vertices with more than $m$ edges joining them. Therefore, $H$ is maximal.

Conversely, note that every graph $G$ with $n$ vertices and at most $m$ parallel edges between any given pair of vertices is a minor of $H$. Therefore, the only minor maximal graph amongst graphs with $n$ vertices, at most $m$ parallel edges between any given pair of vertices, and spectator floor $n-1$ is $H$.
\end{proof}

We now present the first of the two main results of this section, which together provide a complete characterization of the minor maximal graphs with a given spectator floor value. Note that complete graphs are both $1$-saturated $1$-parades and $1$-saturated $2$-parades. This gives some intuition for the reason the first sentence of the theorem is needed.

\begin{theorem}
Let $k\leq n-2$ and $m\geq1$, or let $k=n-1$ and $m\geq2$.
If $G$ is minor maximal amongst graphs with $n$ vertices, at most $m$ parallel edges between any given vertices, and spectator floor $k$, then $G$ is an $m$-saturated crowded $(n-k)$-parade.
\end{theorem}

\begin{proof}
By \cref{lem:1-parade}, the result holds when $k=n-1$. Therefore, we may assume that $k\leq n-2$.

Suppose that $G$ is minor maximal amongst graphs with $n$ vertices, at most $m$ parallel edges between any given vertices, and spectator floor $k$.  We will show that $G$ is an $m$-saturated crowded $(n-k)$-parade.

As was shown in the proof of \cref{lem:maximals}, $\uspc{G}=k$ and so $G$ has a parade $P$ of $n-k\geq2$ vertices.  We will call the endpoints of $P$ by the names $a$ and $b$.  We will now show that $P$ has the properties in the definition of a crowded $(n-k)$-parade, \cref{def:crowded_parade}.  Let $v$ be a vertex outside the parade $P$.  We will consider cases based on the number of vertices in $P$ that are adjacent to $v$.

Suppose $v$ is adjacent to three or more vertices in $P$.  Then two of the parade vertices that $v$ is adjacent to have a distance of 2 or more within the parade.  Hence going through $v$ would provide a path of the same or lesser length between those two vertices as compared to the parade.  This contradicts the properties of a parade.  So $v$ cannot be adjacent to three or more vertices in the parade.

Now suppose that $v$ is adjacent to 2 vertices in $P$, and those two vertices have a distance of 2 or more along the parade.  The same argument from the last paragraph applies; this is a contradiction.  Hence, if $v$ is adjacent to exactly two vertices in the parade, then those two vertices must also be adjacent to each other.

Next, suppose that $v$ is adjacent to exactly 1 vertex in $P$.  There are two cases: either $v$ is adjacent to an endpoint of $P$, or not.

Suppose $v$ is adjacent to an endpoint of the parade; without loss of generality, let us assume $v$ is adjacent to $a$.  Let $x$ be the vertex of $P$ that is adjacent to $a$, and let $H$ be the graph made from $G$ by adding edge $\{v,x\}$, as shown in \cref{fig:maxl_pf_1}.

\begin{figure}[!htbp]
\begin{subfigure}[b]{.47\textwidth}
\[\begin{tikzpicture}[x=1cm, y=1cm]
\vertex[fill,inner sep=1pt,minimum size=1pt] (a) at (0,0)[label=below:$a$]{};
\vertex[fill,inner sep=1pt,minimum size=1pt] (x) at (1,0)[label=below:$x$]{};
\vertex[fill,inner sep=1pt,minimum size=1pt] (y1) at (2,0){};
\vertex[fill,inner sep=1pt,minimum size=1pt] (y2) at (3,0){};
\vertex[fill,inner sep=1pt,minimum size=1pt] (b) at (4,0)[label=below:$b$]{};
\vertex[fill,inner sep=1pt,minimum size=1pt] (v) at (0,1)[label=left:$v$]{};

\draw (a) to (y1);
\draw (y1) to (2.3,0); \draw[dotted,semithick] (2.4,0) to (2.6,0); \draw(2.7,0) to (y2);
%\draw[dotted, thick] (y1) to (y2);
\draw (y2) to (b);
\draw (a) to (v);

\end{tikzpicture}\]
\caption{The subgraph of $G$ induced by $P\cup\{v\}$.}
\end{subfigure}\begin{subfigure}[b]{.47\textwidth}
\[\begin{tikzpicture}[x=1cm, y=1cm]
\vertex[fill,inner sep=1pt,minimum size=1pt] (a) at (0,0)[label=below:$a$]{};
\vertex[fill,inner sep=1pt,minimum size=1pt] (x) at (1,0)[label=below:$x$]{};
\vertex[fill,inner sep=1pt,minimum size=1pt] (y1) at (2,0){};
\vertex[fill,inner sep=1pt,minimum size=1pt] (y2) at (3,0){};
\vertex[fill,inner sep=1pt,minimum size=1pt] (b) at (4,0)[label=below:$b$]{};
\vertex[fill,inner sep=1pt,minimum size=1pt] (v) at (0,1)[label=left:$v$]{};

\draw (a) to (y1);
\draw (y1) to (2.3,0); \draw[dotted,semithick] (2.4,0) to (2.6,0); \draw(2.7,0) to (y2);
\draw (y2) to (b);
\draw (a) to (v);
\draw (v) to (x);

\end{tikzpicture}\]
\caption{The subgraph of $H$ induced by $P\cup\{v\}$.}
\end{subfigure}
\caption{}\label{fig:maxl_pf_1}
\end{figure}

Then by \cref{lem:maximals}, one of the following is true:
\begin{align*}
d_H(a,b)&=d_G(a,v)+1+d_G(x,b)=1+d_G(a,x)+d_G(x,b) \geq 1+d_G(a,b), \ \ \mbox{or} \\
d_H(a,b)&=d_G(a,x)+1+d_G(v,b)=1+d_G(a,v)+d_G(v,b)\geq 1+d_G(a,b).
\end{align*}
However, since $G$ is a subgraph of $H$, we must have $d_H(a,b)\leq d_G(a,b)$.  This is a contradiction.

Now we consider the situation where $v$ is adjacent to exactly one vertex of $P$, and that vertex is not an endpoint of $P$.  Let $w$ be the vertex of $P$ adjacent to $v$, and let $w$ and $x$ be the vertices of $P$ adjacent to $w$, where $x$ is closer to $a$ than $y$ is.  Now consider the graph $X$, which is the same as $G$ but with an edge from $v$ to $x$ added, and the graph $Y$, which is the same as $G$ but with an edge from $v$ to $y$ added.  $G$, $X$, and $Y$ are as shown in \cref{fig:maxl_pf_2}.

\begin{figure}[!htbp]
\centering
\begin{subfigure}[b]{.3\textwidth}
\[\begin{tikzpicture}[x=0.8cm, y=1cm]
\vertex[fill,inner sep=1pt,minimum size=1pt] (a) at (0,0)[label=below:$a$]{};
\vertex[fill,inner sep=1pt,minimum size=1pt] (x) at (1,0)[label=below:$x$]{};
\vertex[fill,inner sep=1pt,minimum size=1pt] (w) at (2,0)[label=below:$w$]{};
\vertex[fill,inner sep=1pt,minimum size=1pt] (y) at (3,0)[label=below:$y$]{};
\vertex[fill,inner sep=1pt,minimum size=1pt] (b) at (4,0)[label=below:$b$]{};
\vertex[fill,inner sep=1pt,minimum size=1pt] (v) at (2,1)[label=left:$v$]{};
\draw (a) to (0.3,0); \draw [dotted,semithick] (0.4,0) to (0.6,0); \draw (0.7,0) to (x);
%\draw[dotted, thick] (a) to (x);
\draw (x) to (y);
\draw (y) to (3.3,0); \draw[dotted,semithick] (3.4,0) to (3.6,0); \draw (3.7,0) to (b);
%\draw[dotted, thick] (y) to (b);
\draw (v) to (w);
\end{tikzpicture}\]
\caption{The subgraph of $G$ induced by $P\cup\{v\}$.}
\end{subfigure}\;\;\;\;\begin{subfigure}[b]{.3\textwidth}
\[\begin{tikzpicture}[x=0.8cm, y=1cm]
\vertex[fill,inner sep=1pt,minimum size=1pt] (a) at (0,0)[label=below:$a$]{};
\vertex[fill,inner sep=1pt,minimum size=1pt] (x) at (1,0)[label=below:$x$]{};
\vertex[fill,inner sep=1pt,minimum size=1pt] (w) at (2,0)[label=below:$w$]{};
\vertex[fill,inner sep=1pt,minimum size=1pt] (y) at (3,0)[label=below:$y$]{};
\vertex[fill,inner sep=1pt,minimum size=1pt] (b) at (4,0)[label=below:$b$]{};
\vertex[fill,inner sep=1pt,minimum size=1pt] (v) at (2,1)[label=left:$v$]{};
\draw (a) to (0.3,0); \draw [dotted,semithick] (0.4,0) to (0.6,0); \draw (0.7,0) to (x);
%\draw[dotted, thick] (a) to (x);
\draw (x) to (y);
\draw (y) to (3.3,0); \draw[dotted,semithick] (3.4,0) to (3.6,0); \draw (3.7,0) to (b);
%\draw[dotted, thick] (y) to (b);
\draw (v) to (w);
\draw (v) to (x);
\end{tikzpicture}\]
\caption{The subgraph of $X$ induced by $P\cup\{v\}$.}
\end{subfigure}\;\;\;\;\begin{subfigure}[b]{.3\textwidth}
\[\begin{tikzpicture}[x=0.8cm, y=1cm]
\vertex[fill,inner sep=1pt,minimum size=1pt] (a) at (0,0)[label=below:$a$]{};
\vertex[fill,inner sep=1pt,minimum size=1pt] (x) at (1,0)[label=below:$x$]{};
\vertex[fill,inner sep=1pt,minimum size=1pt] (w) at (2,0)[label=below:$w$]{};
\vertex[fill,inner sep=1pt,minimum size=1pt] (y) at (3,0)[label=below:$y$]{};
\vertex[fill,inner sep=1pt,minimum size=1pt] (b) at (4,0)[label=below:$b$]{};
\vertex[fill,inner sep=1pt,minimum size=1pt] (v) at (2,1)[label=left:$v$]{};
\draw (a) to (0.3,0); \draw [dotted,semithick] (0.4,0) to (0.6,0); \draw (0.7,0) to (x);
%\draw[dotted, thick] (a) to (x);
\draw (x) to (y);
\draw (y) to (3.3,0); \draw[dotted,semithick] (3.4,0) to (3.6,0); \draw (3.7,0) to (b);
%\draw[dotted, thick] (y) to (b);
\draw (v) to (w);
\draw (v) to (y);
\end{tikzpicture}\]
\caption{The subgraph of $Y$ induced by $P\cup\{v\}$.}
\end{subfigure}
\caption{}\label{fig:maxl_pf_2}
\end{figure}

For $X$, \cref{lem:maximals} tells us that one of two equations must hold.
% of \cref{eqn:QX1,eqn:QX2}, below, is true.
We have a first possibility
\begin{align}\label{eqn:QX1}
\begin{split}
    d_X(a,b)&=d_G(a,v)+1+d_G(x,b)\\
    &=d_G(a,v)+d_G(v,w)+d_G(x,b)\\
    &\geq d_G(a,w)+d_G(x,b)\\
    &=d_G(a,b)+1,
\end{split}
\end{align}
or a second possibility
\begin{align}\label{eqn:QX2}
\begin{split}
    d_X(a,b)&=d_G(a,x)+1+d_G(v,b)\\
    &=d_G(a,x)+1+d_G(x,v)-d_G(x,v)+d_G(v,b)\\
    &>d_G(a,x)+d_G(x,v)+d_G(v,b)-1\\
    &> d_G(a,b)-1.
\end{split}
\end{align}
Note that the second strict inequality in \cref{eqn:QX2} is because an off-parade path between two parade vertices is longer than the parade route.

Since $G$ is a subgraph of $X$, we must have $d_X(a,b)\leq d_G(a,b)$, therefore \cref{eqn:QX1} is a contradiction, and so \cref{eqn:QX2} is true and implies $d_X(a,b)=d_G(a,b)$.

For $Y$, \cref{lem:maximals} tells us that one of the following is true: either
\begin{equation}\label{eqn:QY1}
    d_Y(a,b)=d_G(a,v)+1+d_G(y,b)
\end{equation}
or
\begin{equation}\label{eqn:QY2}
    d_Y(a,b)=d_G(a,y)+1+d_G(v,b).
\end{equation}

However, \cref{eqn:QY2} leads to a contradiction by the same logic as for \cref{eqn:QX1}.  Hence \cref{eqn:QY1} is true and by the same logic as for \cref{eqn:QX2} it implies $d_Y(a,b)=d_G(a,b)$.

When we replace the left-hand sides of \cref{eqn:QX2,eqn:QY1} with $d_G(a,b)$ and add them together, we obtain
\begin{align}\label{eqn:QXQY}
\begin{split}
    2d_G(a,b)&=d_G(a,v)+d_G(v,b)+d_G(a,x)+d_G(y,b)+2\\
    &>d_G(a,b)+d_G(a,x)+d_G(y,b)+2\implies\\
    d_G(a,b)&>d_G(a,x)+d_G(y,b)+2\\
    &=d_G(a,b).
\end{split}
\end{align}
(Note that strict inequality is required because an off-parade path between two parade vertices is longer than the parade route.)
This yields a contradiction.

Thus we conclude that in $G$, a vertex $v$ that is not in $P$ cannot be adjacent to exactly one vertex in $P$.

For the last case, let us consider vertices of $G$ that are not in $P$ and are not adjacent to any vertices in $P$.  We will show these cannot exist.  First note that $G$ must be connected because of \cref{lem:maxls_are_connected}. Let $d_G(v,P)$ denote the distance from any vertex $v$ to the path $P$; we define
\[d_G(v,P):=\min_{p\in P}\{d_G(v,p)\}.\]
Since $G$ is connected, $d_G(v,P)$ is finite for all $v\in G$.  For any vertex $v$ that is not in $P$ and not adjacent to $P$, either $d_G(v,P)=2$ or $d_G(v,P)>2$.  If $d_G(v,P)>2$, then there is a path of length $d_G(v,P)$ from $v$ to a vertex $p\in P$, and so there is some vertex $v'$ along that path that has $d_G(v',P)=2$.

Hence, if there are any vertices of $G$ that are not in $P$ and are not adjacent to $P$, then there is some vertex $v$ with $d_G(v,P)=2$.  We will show this leads to a contradiction.  Since $d_G(v,P)=2$, there is some vertex, call it $w$, that is adjacent to $v$ and adjacent to $P$.  We have already shown in this proof that if $w$ is adjacent to any vertex of $P$, then it is adjacent to exactly two vertices of $P$, which are also adjacent to each other.  Let us call the two vertices of $P$ that $w$ is adjacent to by the names $x$ and $y$, and the endpoints of $P$ by the names $a$ and $b$.  As before, we shall assume the $a$ is the endpoint that is closer to $x$.

Now consider the graph $X$, which is the same as $G$ but with an edge from $v$ to $x$ added, and the graph $Y$, which is the same as $G$ but with an edge from $v$ to $y$ added.  The situation is shown in \cref{fig:maxl_pf_3}.

\begin{figure}[!htbp]
\centering
\begin{subfigure}[b]{.3\textwidth}
\[\begin{tikzpicture}[x=1cm, y=1cm]
\vertex[fill,inner sep=1pt,minimum size=1pt] (a) at (0,0)[label=below:$a$]{};
\vertex[fill,inner sep=1pt,minimum size=1pt] (x) at (0.85,0)[label=below:$x$]{};
\vertex[fill,inner sep=1pt,minimum size=1pt] (w) at (1.5,1)[label=right:$w$]{};
\vertex[fill,inner sep=1pt,minimum size=1pt] (y) at (2.15,0)[label=below:$y$]{};
\vertex[fill,inner sep=1pt,minimum size=1pt] (b) at (3,0)[label=below:$b$]{};
\vertex[fill,inner sep=1pt,minimum size=1pt] (v) at (1.5,2)[label=right:$v$]{};

\draw (a) to (0.3,0); \draw [dotted,semithick] (.33,0) to (.52,0); \draw (.55,0) to (x); 
%\draw[dotted, thick] (a) to (x);
\draw (y) to (2.45,0); \draw[dotted,semithick] (2.48,0) to (2.67,0); \draw (2.7,0) to (b);
%\draw[dotted, thick] (y) to (b);
\draw (v) to (w) to (x) to (y) to (w);
\end{tikzpicture}\]
\caption{The subgraph of $G$ induced by $P\cup\{v,w\}$.}
\end{subfigure}\;\;\;\;\begin{subfigure}[b]{.3\textwidth}
\[\begin{tikzpicture}[x=1cm, y=1cm]
\vertex[fill,inner sep=1pt,minimum size=1pt] (a) at (0,0)[label=below:$a$]{};
\vertex[fill,inner sep=1pt,minimum size=1pt] (x) at (0.85,0)[label=below:$x$]{};
\vertex[fill,inner sep=1pt,minimum size=1pt] (w) at (1.5,1)[label=right:$w$]{};
\vertex[fill,inner sep=1pt,minimum size=1pt] (y) at (2.15,0)[label=below:$y$]{};
\vertex[fill,inner sep=1pt,minimum size=1pt] (b) at (3,0)[label=below:$b$]{};
\vertex[fill,inner sep=1pt,minimum size=1pt] (v) at (1.5,2)[label=right:$v$]{};

\draw (a) to (0.3,0); \draw [dotted,semithick] (.33,0) to (.52,0); \draw (.55,0) to (x); 
%\draw[dotted, thick] (a) to (x);
\draw (y) to (2.45,0); \draw[dotted,semithick] (2.48,0) to (2.67,0); \draw (2.7,0) to (b);
%\draw[dotted, thick] (y) to (b);
\draw (v) to (w) to (x) to (y) to (w);
\draw (v) to (x);
\end{tikzpicture}\]
\caption{The subgraph of $X$ induced by $P\cup\{v,w\}$.}
\end{subfigure}\;\;\;\;\begin{subfigure}[b]{.3\textwidth}
\[\begin{tikzpicture}[x=1cm, y=1cm]
\vertex[fill,inner sep=1pt,minimum size=1pt] (a) at (0,0)[label=below:$a$]{};
\vertex[fill,inner sep=1pt,minimum size=1pt] (x) at (0.85,0)[label=below:$x$]{};
\vertex[fill,inner sep=1pt,minimum size=1pt] (w) at (1.5,1)[label=left:$w$]{};
\vertex[fill,inner sep=1pt,minimum size=1pt] (y) at (2.15,0)[label=below:$y$]{};
\vertex[fill,inner sep=1pt,minimum size=1pt] (b) at (3,0)[label=below:$b$]{};
\vertex[fill,inner sep=1pt,minimum size=1pt] (v) at (1.5,2)[label=left:$v$]{};

\draw (a) to (0.3,0); \draw [dotted,semithick] (.33,0) to (.52,0); \draw (.55,0) to (x); 
%\draw[dotted, thick] (a) to (x);
\draw (y) to (2.45,0); \draw[dotted,semithick] (2.48,0) to (2.67,0); \draw (2.7,0) to (b);
%\draw[dotted, thick] (y) to (b);
\draw (v) to (w) to (x) to (y) to (w);
\draw (v) to (y);
\end{tikzpicture}\]
\caption{The subgraph of $Y$ induced by $P\cup\{v,w\}$.}
\end{subfigure}
\caption{}\label{fig:maxl_pf_3}
\end{figure}

These graphs are not the same as the previous $X$ and $Y$, but all of the arguments from before regarding \cref{eqn:QX1,eqn:QX2,eqn:QY1,eqn:QY2} still apply.  Most of \cref{eqn:QXQY} applies as well, with the exception of the last line, where we used the fact that $d_G(a,x)+d_G(y,b)+2=d_G(a,b)$, which is no longer true.  Now we have $d_G(a,x)+d_G(y,b)+2=d_G(a,b)+1$, which still leads to a contradiction when applied to the penultimate line of \cref{eqn:QXQY}.

Thus we conclude that there are no vertices of $G$ that are not in $P$ and not adjacent to $P$.  Every vertex of $G$ which is not in $P$ must be adjacent to two adjacent vertices in $P$.  This fulfills the first condition for $G$ to be a crowded $(n-k)$-parade.  We will now show that the second condition is also true.

Let $v$ and $v'$ be two vertices outside $P$ that are adjacent to a common vertex $w$ in $P$, and suppose $v,v'$ are not adjacent to each other.  Let $F$ be the graph made from $G$ by adding an edge between $v$ and $v'$.  Now we can assume without loss of generality that the first equation from \cref{lem:maximals} is true (if the second equation is the true one, just swap the names of $a$ and $b$).  Then we have  
\begin{align*}\label{eqn:Fdist}
d_F(a,b)&=d_G(a,v)+1+d_G(v',b)\\
&=d_G(a,v)+d_G(v,w)-d_G(v,w)+1-d_G(w,v')+d_G(w,v')+d_G(v',b)\\
&=d_G(a,v)+d_G(v,w)-1+d_G(w,v')+d_G(v',b)\\
&\geq d_G(a,w)+d_G(w,b)+1\\
&=d_G(a,b)+1.
\end{align*}
(Note that to get from line 3 to line 4, we are applying the fact that an off-parade path between two parade vertices is strictly longer than the parade route.)

However, $G$ is a subgraph of $F$, so we also have $d_F(a,b)\leq d_G(a,b)$.  This is a contradiction.  Hence, $G$ is a crowded $(n-k)$-parade.

We will now argue that $G$ is  $m$-saturated.  For any edge $e$ of $G$ that is not in $P$, if there are fewer than $m$ parallel copies of $e$, then we can add another parallel copy of $e$ without changing any distances in $G$, and without creating any new paths between the endpoints of $P$ that have the same length as $P$.  Hence, adding such an edge does not change the spectator floor value of the graph.  Given that $G$ is maximal amongst graphs with $n$ vertices, at most $m$ parallel copies of each edge, and spectator floor $k$, then, $G$ must already contain all such edges.  Hence $G$ is an $m$-saturated crowded $(n-k)$-parade.\end{proof}

We will now present the second of the two main results in this section, which is the converse of the last theorem, showing that we have a complete characterization of the minor maximal graphs.

\begin{theorem}\label{thm:maxl_2}
Let $k\leq n-2$ and $m\geq1$, or let $k=n-1$ and $m\geq2$.
If $G$ is an $m$-saturated crowded $(n-k)$-parade, then $G$ is minor maximal amongst graphs with $n$ vertices, at most $m$ parallel edges between any given vertices, and spectator floor $k$.
\end{theorem}
\begin{proof}
By \cref{lem:1-parade}, the result holds when $k=n-1$. Therefore, we may assume that $k\leq n-2$.

Suppose that $G$ is an $m$-saturated crowded $(n-k)$-parade, but $G$ is not minor maximal amongst graphs with $n$ vertices, at most $m$ parallel edges between any given vertices, and spectator floor $k$.  Then there is some other graph, call it $G'$, which is minor maximal on that set and of which $G$ is a proper subgraph.  Since $G'$ is minor maximal, it must be an $m$-saturated $(n-k)$-parade as well.

Let us call the parade for which $G$ fulfills the definition of a crowded parade, \cref{def:crowded_parade}, by the name $P$.  First note that if $m>1$, then there are no unique paths in $G'$ except $P$ and its subpaths, hence for $m>1$, $P$ is still the parade for which $G'$ is a crowded parade.

Now suppose $m=1$ (in other words, we are working with simple graphs only) and suppose $P$ is either not a parade in $G'$ or not a parade for which $G'$ is crowded.

Let us call the parade for which $G'$ is a crowded parade by the name $P'$, and the endpoints of $P$ by the names $a,b$ and of $P'$ by the names $a',b'$.  Since $P\neq P'$, at least one of $a',b'$ is not in $P$.  Since $G$ is a subgraph of $G'$, $d_{G'}(v,w)\leq d_G(v,w)$ for all pairs of vertices $v,w$.  So $d_G(a',b')\geq d_{G'}(a',b')=d_G(a,b)=n-k-1$.

A generalized figure of graph $G$ is shown in \cref{fig:gen_crowded_parade}, where path $P$ consists of the vertices labeled $v_1=a$ to $v_{n-k}=b$.  Nodes labeled $K_{m_i}$ represent complete subgraphs on $m_i$ vertices, and bold edges represent a complete set of edges between the connected structures.  It should be understood, however, that it is possible for some $m_i$ to be 0, in which case $K_{m_i}$ is an empty graph, and there are no edges connecting $K_{m_{i-1}}$, $K_{m_{i}}$, and $K_{m_{i+1}}$.  For example, in the previous \cref{fig:crowded_parade}, we had $\{m_1,m_2,m_3,m_4,m_5,m_6\}=\{0,1,2,1,0,3\}$.

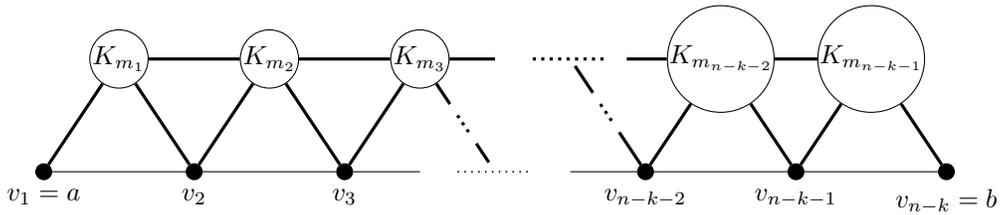
\begin{figure}[!htbp]
\[\begin{tikzpicture}[x=1cm, y=1cm]
\vertex[fill] (a) at (0,0)[label=below:${v_1=a}$]{};
\vertex[fill] (b) at (2,0)[label=below:$v_2$]{};
\vertex[fill] (c) at (4,0)[label=below:$v_3$]{};
\node (d) at (6,0){};
\vertex[fill] (e) at (8,0)[label=below:$v_{n-k-2}$]{};
\vertex[fill] (f) at (10,0)[label=below:$v_{n-k-1}$]{};
\vertex[fill] (g) at (12,0)[label=below:${v_{n-k}=b}$]{};

\vertex[] (ab) at (1,1.5){$K_{m_1}$};
\vertex[] (bc) at (3,1.5){$K_{m_2}$};
\vertex[] (cd) at (5,1.5){$K_{m_3}$};
\vertex[] (ef) at (9,1.5){$K_{m_{n-k-2}}$};
\vertex[] (fg) at (11,1.5){$K_{m_{n-k-1}}$};
\draw (a) to (c);
\draw(c) to (5,0); \draw[dotted,semithick] (5.5,0) to (6.5,0); \draw (7,0) to (e);
%\draw[dotted,semithick] (c) to (e);
\draw (e) to (g);
\draw[very thick] (ab) to (bc) to (cd);
\draw[very thick](cd) to (6,1.5); \draw[very thick, dotted] (6.5,1.5) to (7.5,1.5); \draw[very thick] (7.75,1.5) to (ef); 
%\draw[very thick,dotted] (cd) to (ef);
\draw[very thick] (ef) to (fg);
\draw[very thick] (a) to (ab) to (b) to (bc) to (c) to (cd);
\draw[very thick] (e) to (ef) to (f) to (fg) to (g);
\draw[very thick] (cd) to (5.4,.9); \draw[very thick, dotted] (5.5, .75) to (5.65,.525); \draw[very thick] (5.75,.375) to (d); 
%\draw[dotted,very thick] (cd) to (d);
%\draw[dotted,very thick] (de) to (e);
\draw[very thick] (de) to (7.3,1.05); \draw[very thick, dotted] (7.4,.9) to (7.55, .675); \draw[very thick] (7.65,.525) to (e);
 
\end{tikzpicture}\]
\caption{A generalized figure of $G$, a simple crowded $(n-k)$-parade.}
\label{fig:gen_crowded_parade}
\end{figure}

For graph $G$, since $d_G(a',b')\geq n-k-1$, it must be the case that either $a'$ or $b'$ is in $K_{m_1}$ or $K_{m_{n-k-1}}$, which are on opposite ends of the graph.  Furthermore, in order for the path $P'$ from $a'$ to $b'$ to be a unique shortest path in $G'$, it is necessary that $K_{m'_2}$ and $K_{m'_{n-k-2}}$ be empty, where by $K_{m'_i}$ we mean the subgraph in $G'$ corresponding to $K_{m_i}$ in $G$.  If these are not empty, then there would be many alternate routes of equal length connecting $a'$ to $b'$.

Since $K_{m'_2}$ and $K_{m'_{n-k-2}}$ are empty, and $G'$ is simple, there are graph symmetries identifying any vertex in $K_{m'_1}$ with any other vertex in $K_{m'_1}$ as well as $v_1$.  Likewise, there are graph symmetries identifying any vertex in $K_{m'_{n-k-1}}$ with any other vertex in $K_{m'_{n-k-1}}$ and $v_{n-k}$.  Thus, there is a graph symmetry identifying $P$ with $P'$. Therefore, $P$ is still a parade in $G'$. Moreover, we can assume that $P$ is the parade for which $G'$ is a crowded parade.

In any case, we now have that both $G$ and $G'$ are crowded $(n-k)$-parades on the same number of vertices and are crowded on the same parade $P$, and $G$ is a proper subgraph of $G'$.  Thus $G'$ contains at least one edge $e$ that is not in $G$.  There are three cases: either $e$ connects two vertices in $P$, or it connects a vertex in $P$ to a vertex not in $P$, or it connects two vertices not in $P$.

If $e$ connects two vertices in $P$, then $P$ would not be a parade of $n-k$ vertices in $G'$, so this is a contradiction.  If $e$ connects a vertex in $P$ to a vertex not in $P$, then this provides an alternate path of the same or shorter length between the endpoints of $P$, in contradiction to $P$ being a parade.  Finally, if $e$ connects two vertices not in $P$, this also provides an alternate path of the same or shorter length between the endpoints of $P$.  In any case, we get a contradiction.  The theorem is thus proven.
\end{proof}

\section{Further Questions}
\label{further questions}
There are many additional questions that one may consider in this line of research. In this paper, we have characterized the minor-minimal graphs $G$ with $\uspcf{G}=k$ for $k=1$ and $k=2$. For larger, $k$, consider the following.

\begin{question}

Can we characterize the minor-minimal graphs $G$ with $\uspcf{G}=k$ for $k=3$?  $k=4$?  Etc.
\end{question}

Algorithmic questions related to the spectator number and spectator floor of a graph have not been considered in this paper, but we hope to work on some of these questions in the future.

\begin{question}
Can the minor-monotone floor of the spectator number be computed in polynomial time?
\end{question}

\begin{question}
In the algorithm for calculating the minor-monotone floor of the spectator number, what are the optimal edges to add?
\end{question}

\begin{question}
Can the minor-monotone floor of the spectator number be computed with a ``greedy algorithm"? (That is, can we add a set of edges to a graph $G$ to obtain a supergraph $H$ such that each time we add an edge the spectator number is weakly decreasing? strictly decreasing?) 
\end{question} 
%\kevin{I'll try to think about whether we can interpret this as a matroid question.

%Update 1: I've thought about this some. While I don't think there's a way to interpret this as a matroid question, I did come up with a counterexample for the ``strictly'' part of this question. But I'm guessing we don't want to include that in this paper.

%Update 2: I think the answer to the ``weakly'' part of this question is yes. I think I have a proof sketch, but I'd rather not take the time to put it in this paper, so we can get it submitted sooner. Again, I didn't use matroids.}

A related, but distinct question is the following.

\begin{question}
If $\uspc{G}=k$ and $\uspcf{G}=m$, can we always find a supergraph $F$ of $G$ that achieves $\uspc{F}=i$ for all $i$ such that $m\leq i \leq k$?
\end{question}

Finally, one can consider how well this bound relates to our original motivation.

\begin{question}
\textcolor{black}{For what graphs $G$ do we have $\uspc{G}=n-q(G)$?}

%This problem originated as a bound on the inverse eigenvalue problem.  How good of a bound is this, and when does it fail? Etc.
\end{question}

\section*{Acknowledgements}
This work started at the MRC workshop “Finding Needles in Haystacks: Approaches to Inverse
Problems using Combinatorics and Linear Algebra”, which took place in June 2021 with support from the
National Science Foundation and the American Mathematical Society. The authors are grateful to the
organizers of this meeting.

In particular, this material is based upon work supported by the National Science Foundation under Grant Number DMS 1916439.

{\color{black}The authors are also grateful to the referees for their helpful comments.}
\bibliography{Revision_of_L210}{}
\bibliographystyle{plain}

\end{document}